\documentclass[12pt]{amsart}

\usepackage[colorlinks,
linkcolor=blue,
anchorcolor=blue,
citecolor=blue]{hyperref}
\usepackage{epsfig}
\usepackage{graphicx}
\usepackage{amssymb, amstext, amscd, amsmath}
\usepackage{amsthm, mathrsfs, amsfonts,dsfont}
\usepackage{fullpage}
\usepackage{txfonts}
\usepackage{fancybox}
\usepackage{color}
\usepackage{cite}
\usepackage{comment}
\usepackage{tikz-cd} 

\allowdisplaybreaks

\newtheorem{theorem}{Theorem}[section]
\newtheorem{lemma}{Lemma}[section]
\newtheorem{proposition}{Proposition}[section]
\newtheorem{corollary}{Corollary}[section]

\theoremstyle{definition}

\newtheorem{remark}{Remark}[section]

\numberwithin{equation}{section}

\begin{document}
\title{\fontsize{14}{17}\selectfont
Asymptotic Rigidity and Boundary Structure of Yang's Numerical
Invariants for the Bidisk Submodules $[z^k-w^\ell]$}
\author{Yin Liu}
\address{School of Mathematics and Statistics, Nanyang Normal University,
Nanyang, Henan, 473061, P. R. China} 
\email{lylight@mail.bnu.edu.cn}

\author{Yufeng Lu}
\address{School of Mathematical Sciences, Dalian University of Technology,
Dalian, Liaoning, 116024, P. R. China}
\email{lyfdlut@dlut.edu.cn}

\author{Yixin Yang*}
\address{School of Mathematical Sciences, Dalian University of Technology,
Dalian, Liaoning, 116024, P. R. China}
\email{yangyixin@dlut.edu.cn}

\thanks{*Corresponding author.}

\begin{abstract}
Let
\[
M_{k,\ell}=[z^k-w^\ell]\subset H^2(\mathbb D^2),
\qquad k,\ell\in\mathbb N,\quad k\ne\ell .
\]
We study the large-index asymptotics of Yang's numerical invariants
and the boundary structure of their generating function. Starting from the exact staircase formula
obtained in our preceding work, we prove that
\[
\Sigma_j(M_{k,\ell})
=
\frac{C_{k,\ell}}{j}
+
O(j^{-2}),
\]
where
\[
C_{k,\ell}
=
\int_0^\infty
\frac{x^2}
{(x+1/k)^2(x+1/\ell)^2}\,dx .
\]
The first strict descent together with $C_{k,\ell}$ recovers the
unordered pair $\{k,\ell\}$, yielding an asymptotic rigidity principle. At the next order we obtain a periodic correction
\[
\Sigma_j(M_{k,\ell})
=
\frac{C_{k,\ell}}{j}
+
\frac{\Psi_{k,\ell}(j)}{j^2}
+
O(j^{-3}),
\]
whose least period is \(\operatorname{lcm}(k,\ell)\), and we determine
the leading amplitudes of the strict drops.

For Yang's generating function
\[
\mathcal P_{k,\ell}(t)
=
\sum_{j=0}^{\infty}\Sigma_j(M_{k,\ell})t^j,
\]
we prove that it has radius of convergence one and a
logarithmic singularity at $t=1$; hence it is never a polynomial
for $k\ne\ell$, giving a negative answer to Yang's polynomiality
question within this quasi-homogeneous family. The periodic
higher-order corrections generate root-of-unity polylogarithmic
boundary modes. The leading logarithmic coefficient together with
the second-order boundary support determines the unordered pair
$\{k,\ell\}$, while successive renormalized boundary limits recover
the Fourier coefficients of every finite-order periodic asymptotic
term.
\end{abstract}

\maketitle

\setcounter{tocdepth}{2}
\tableofcontents

\section{Introduction}

The Hardy space $H^2(\mathbb D^2)$ over the bidisk is naturally a
Hilbert module over the polynomial ring $\mathbb C[z,w]$, with the
module action given by multiplication by the coordinate functions.
The structure of its submodules is substantially more complicated
than in the one-variable Hardy space, and operator-theoretic
invariants associated with the restricted coordinate shifts provide
a useful means of distinguishing and organizing such submodules;
see, for example, \cite{DouglasPaulsen1989,ChenGuo2003,GuoYang2004,
Yang2001,Yang2005Core,Yang2005HS}.

A sequence of numerical invariants for submodules of
$H^2(\mathbb D^2)$ was introduced by Yang in \cite{Yang2001}.
If $M\subset H^2(\mathbb D^2)$ is a submodule and
$\{\phi_n\}\subset M\ominus zM$ and
$\{\psi_m\}\subset M\ominus wM$ are orthonormal bases of the two
defect spaces, then the higher numerical invariants may be written as
\[
\Sigma_j(M)
=
\sum_{n=0}^{\infty}\sum_{m=0}^{\infty}
\left|
\left\langle
w^j\phi_n,z^j\psi_m
\right\rangle
\right|^2,
\qquad j\geq0.
\]

These quantities are intrinsic to the module and are closely related
to the core operator, cross-commutators, and Hilbert--Schmidt
properties. In the same paper, Yang introduced the formal power series
\[
p_M(t)=\sum_{j=0}^{\infty}\Sigma_j(M)t^j
\]
as a unified invariant. He also proved that $p_M(t)$ is constant
precisely when $M$ is unitarily equivalent to the full Hardy module
\cite[Sec.~3, Theorem~3.5, p.~537]{Yang2001}.

Explicit determination of the higher invariants is difficult even for
polynomially generated submodules.  For homogeneous principal
submodules, Azari Key, Lu and Yang developed a Toeplitz-determinant
method which reduces the calculation to the Gram matrices of the
homogeneous wandering spaces \cite{AzariLuYang2017}.  More recently,
the quadratic model $[(z-w)^2]$ was computed explicitly
\cite{LiuLuZu2026Quadratic}, the family $[z^k-w^k]$ was shown to
exhibit a block-repetition phenomenon
\cite{LiuLuZu2026Block}, and strict monotonicity was established for
the repeated-factor family $[(z-w)^k]$
\cite{LiuLuZu2026Strict}.  A weighted OPUC--CMV framework for general homogeneous principal
submodules was subsequently developed in \cite{LuZu2026CMV}.
Within this broader framework, examples were obtained showing that
Yang's numerical invariant sequence need not be monotone for a
general homogeneous generator. This indicates that the behavior of
the complete invariant sequence is sensitive to the algebraic
structure of the generating polynomial.

Quasi-homogeneous quotient modules over the bidisk had previously been
studied from the viewpoints of essential normality, \(K\)-homology, and
trace formulas; see, for example, \cite{GuoWang2012,GuoWangZhang2012}.
In particular, the weighted circle action associated with
\(z^k-w^\ell\) already appears in \cite[Example~5.13]{GuoWangZhang2012}.

Our preceding work \cite{LiuLuZu2026Quasi} determined the complete numerical invariant
sequence for
\[
M_{k,\ell}=[z^k-w^\ell]\subset H^2(\mathbb D^2),
\qquad k,\ell\in\mathbb N,\quad k\ne\ell .
\]
In particular,
\[
\Sigma_j(M_{k,\ell})
=
F\!\left(
\left\lceil\frac{j}{\ell}\right\rceil,
\left\lceil\frac{j}{k}\right\rceil
\right),
\qquad j\ge1,
\]
where $F$ is recalled in Section~\ref{sec2}, and the strict descent set is $\mathcal D_{k,\ell}=k\mathbb N\cup\ell\mathbb N$.
Thus the sequence has a two-parameter arithmetic staircase structure,
and the complete sequence determines the unordered pair $\{k,\ell\}$.

The first motivation for the present paper comes directly from this
exact description. Although the complete invariant sequence and its arithmetic staircase
structure are known, it is natural to ask how much of the parameter
information is already visible in the large-index asymptotic tail.

The second motivation comes from Yang's generating function, which
encodes the numerical invariants $\Sigma_j(M)$ as its coefficients.
This naturally raises a complementary question: how is the large-index
asymptotic structure of the invariant sequence reflected in the
boundary behavior of the generating function? This viewpoint is further
motivated by a question posed by Yang in his original work. For Yang's generating function
\[
p_M(t)=\sum_{j=0}^{\infty}\Sigma_j(M)t^j,
\]
he asked \cite[Sec.~3, p.~537]{Yang2001}:

\medskip
\noindent\textbf{Question (Yang, 2001).}
\emph{When is \(p_M(t)\) a polynomial?}
\medskip

For the quasi-homogeneous family considered here, the corresponding
generating function $\mathcal P_{k,\ell}:=p_{M_{k,\ell}}$ has a
logarithmic singularity at $t=1$ and hence is never a polynomial. More importantly for the present work, its boundary singularities
encode the periodic higher-order asymptotic structure of the
coefficient sequence.

These two related motivations lead to the following three questions,
which also indicate the order of the analysis developed below:
\begin{enumerate}
\item[(i)]
What is the precise decay rate of
$\Sigma_j(M_{k,\ell})$ as $j\to\infty$?

\item[(ii)]
How much of the pair $\{k,\ell\}$ can be recovered from coarse
asymptotic data rather than from the entire invariant sequence?

\item[(iii)]
How is the asymptotic structure of the invariant sequence
reflected in the boundary behavior of Yang's generating function?
\end{enumerate}

Our first main theorem establishes the leading decay law and shows
that the leading tail coefficient, together with the first strict
descent, already determines the two quasi-homogeneous
parameters \(k\) and \(\ell\).

\begin{theorem}[Leading asymptotics and rigidity]
\label{thm:intro-asymptotic-rigidity}
Let
\[
M_{k,\ell}=[z^k-w^\ell],\qquad k,\ell\in\mathbb N,\ k\neq\ell .
\]
Then
\[
\Sigma_j(M_{k,\ell})
=
\frac{C_{k,\ell}}{j}
+O(j^{-2}),
\qquad j\to\infty,
\]
where
\[
C_{k,\ell}
=
\int_0^\infty
\frac{x^2}
{(x+\frac1k)^2(x+\frac1\ell)^2}\,dx .
\]
Moreover, if $a=\min\{k,\ell\},\qquad b=\max\{k,\ell\}$,
then the first strict descent determines \(a\), and for fixed \(a\),
the coefficient \(C_{a,b}\) determines \(b\).
Hence the asymptotic data
\[
\left(\min\mathcal D_{k,\ell},
\lim_{j\to\infty}j\Sigma_j(M_{k,\ell})\right)
\]
determine the unordered pair \(\{k,\ell\}\).
\end{theorem}

Thus Theorem~\ref{thm:intro-asymptotic-rigidity} replaces the full-sequence recovery in ~\cite{LiuLuZu2026Quasi}
by the two coarse data $\min\mathcal  D_{k,\ell}$ and $C_{k,\ell}$.
The ceiling errors enter at the next asymptotic scale, producing
periodic corrections that also govern the boundary structure of
Yang's generating function.

\begin{theorem}[Boundary structure and reconstruction]
\label{thm:intro-boundary-reconstruction}
For the generating function
\[
\mathcal P_{k,\ell}(t)
:=
p_{M_{k,\ell}}(t)
=
\sum_{j=0}^{\infty}\Sigma_j(M_{k,\ell})t^j,
\]
the coefficient asymptotic expansion determines a hierarchy of
polylogarithmic boundary modes.

The leading term
\[
\Sigma_j(M_{k,\ell})
=
C_{k,\ell}j^{-1}+O(j^{-2})
\]
produces the logarithmic singularity at \(t=1\),
\[
\mathcal P_{k,\ell}(t)
=
C_{k,\ell}
\log\frac1{1-t}
+K_{k,\ell}+o(1),
\]
where $K_{k,\ell}$ is a finite constant.

Furthermore, the periodic higher-order coefficients give rise to
root-of-unity polylogarithmic boundary modes. The leading logarithmic
coefficient together with the second-order boundary support determines
the unordered pair $\{k,\ell\}$, while the boundary limits of the
successively renormalized generating functions recover the Fourier
coefficients of every finite-order periodic asymptotic term.
\end{theorem}

The paper is organized as follows. Section~\ref{sec2} develops the integral
and scaling preliminaries for $F$. Section~\ref{sec3} establishes the leading
asymptotics and the asymptotic rigidity principle. Section~\ref{sec4} analyzes
the periodic second-order correction and the leading descent
amplitudes. Section~\ref{sec5} studies Yang's generating function and develops
the higher-order polylogarithmic boundary hierarchy. Section~\ref{sec6} establishes boundary rigidity and reconstructs the finite-order
periodic asymptotic hierarchy from boundary data.

\section{Integral representations and asymptotic preliminaries}
\label{sec2}

Throughout this paper, we use $\mathbb N=\{1,2,3,\ldots\}$,
and let
\[
M_{k,\ell}
=
[z^k-w^\ell]
\subset H^2(\mathbb D^2),
\qquad
k,\ell\in\mathbb N,\quad k\ne\ell.
\]

We shall use the exact formula for the numerical invariant sequence
obtained in \cite{LiuLuZu2026Quasi}.  For $j\geq1$, set
\begin{equation*}
\label{eq:ceiling-parameters}
A_j
=
\left\lceil\frac{j}{\ell}\right\rceil,
\qquad
B_j
=
\left\lceil\frac{j}{k}\right\rceil.
\end{equation*}
Using \cite[Sec.~3.3, Eq.~(3.15)]{LiuLuZu2026Quasi}, we have
\begin{equation}
\label{eq:Sigma-F-recalled}
\Sigma_j(M_{k,\ell})
=
F(A_j,B_j),
\qquad j\geq1,
\end{equation}
where
\begin{equation}
\label{eq:F-definition}
F(\mu,\nu)
:=
\sum_{r=0}^{\infty}
\frac{(r+1)^2}
{(r+\mu)(r+\mu+1)(r+\nu)(r+\nu+1)},
\qquad
\mu,\nu\geq1.
\end{equation}
The function $F$ is symmetric:$
F(\mu,\nu)=F(\nu,\mu)$.

The purpose of this section is to develop forms of
\eqref{eq:F-definition} that are adapted to large-parameter
asymptotics.  In particular, we isolate a continuous scaling function
which will determine the leading coefficient of the numerical
invariant sequence.

\subsection{An integral representation of \texorpdfstring{\(F\)}{F}}

The following elementary integral representation will be convenient
for the asymptotic analysis.

\begin{lemma}
\label{lem:F-integral}
For all $\mu,\nu\geq1$,
\begin{equation}
\label{eq:F-double-integral}
F(\mu,\nu)
=
\int_0^1\!\!\int_0^1
x^{\mu-1}y^{\nu-1}(1-x)(1-y)
\frac{1+xy}{(1-xy)^3}
\,dx\,dy.
\end{equation}
\end{lemma}

\begin{proof}
For every $r\geq0$ and $\mu\geq1$,
\[
\int_0^1 x^{r+\mu-1}(1-x)\,dx
=
\frac{1}{r+\mu}
-
\frac{1}{r+\mu+1}
=
\frac{1}{(r+\mu)(r+\mu+1)}.
\]
Similarly,
\[
\int_0^1 y^{r+\nu-1}(1-y)\,dy
=
\frac{1}{(r+\nu)(r+\nu+1)}.
\]

Hence each summand in \eqref{eq:F-definition} may be written as
\[
\begin{aligned}
\frac{(r+1)^2}
{(r+\mu)(r+\mu+1)(r+\nu)(r+\nu+1)}
=
(r+1)^2
\int_0^1\!\!\int_0^1
x^{r+\mu-1}y^{r+\nu-1}
(1-x)(1-y)
\,dx\,dy.
\end{aligned}
\]
All terms are nonnegative, so Tonelli's theorem gives
\[
\begin{aligned}
F(\mu,\nu)
&=
\int_0^1\!\!\int_0^1
x^{\mu-1}y^{\nu-1}(1-x)(1-y)
\sum_{r=0}^{\infty}(r+1)^2(xy)^r
\,dx\,dy.
\end{aligned}
\]

For \(0\leq t<1\),
\[
\sum_{r=0}^{\infty}(r+1)^2t^r
=
\frac{1+t}{(1-t)^3}.
\]
Thus, for \(xy<1\),
\[
\sum_{r=0}^{\infty}(r+1)^2(xy)^r
=
\frac{1+xy}{(1-xy)^3}.
\]
Since \(xy=1\) in \([0,1]^2\) only at the single point
\((1,1)\), which has measure zero, substituting this identity into
the preceding integral gives \eqref{eq:F-double-integral}.
\end{proof}


The integral representation \eqref{eq:F-double-integral} shows that
large values of $\mu$ and $\nu$ force the main contribution to come
from a neighborhood of $(1,1)$.  The corresponding continuous
scaling limit is described by the following function.

\subsection{The continuous scaling function}

For $\alpha,\beta>0$, define
\begin{equation}
\label{eq:C-alpha-beta}
\mathscr C(\alpha,\beta)
=
\int_0^\infty
\frac{x^2}
{(x+\alpha)^2(x+\beta)^2}
\,dx.
\end{equation}
The integral is finite. Indeed,
\[
0\leq
\frac{x^2}{(x+\alpha)^2(x+\beta)^2}
\leq
\begin{cases}
\dfrac{x^2}{\alpha^2\beta^2}, & 0<x\leq1,\\[6pt]
\dfrac{1}{x^2}, & x\geq1.
\end{cases}
\]
Both comparison functions are integrable on their respective
intervals.

\begin{lemma}
\label{lem:C-explicit}
The function $\mathscr C$ is positive and symmetric.  If
$0<\alpha<\beta$, then
\begin{equation}
\label{eq:C-explicit}
\mathscr C(\alpha,\beta)
=
\frac{\alpha+\beta}{(\beta-\alpha)^2}
-
\frac{2\alpha\beta}{(\beta-\alpha)^3}
\log\frac{\beta}{\alpha}.
\end{equation}
On the diagonal,
\begin{equation}
\label{eq:C-diagonal}
\mathscr C(\alpha,\alpha)
=
\frac{1}{3\alpha}.
\end{equation}
Moreover, for every $c>0$,
\begin{equation}
\label{eq:C-homogeneity}
\mathscr C(c\alpha,c\beta)
=
\frac{1}{c}\mathscr C(\alpha,\beta).
\end{equation}
\end{lemma}

\begin{proof}
Positivity and symmetry follow immediately from
\eqref{eq:C-alpha-beta}.  Assume first that $0<\alpha<\beta$.
Writing $d=\beta-\alpha$, a direct partial fraction decomposition gives
\[
\begin{aligned}
\frac{x^2}
{(x+\alpha)^2(x+\beta)^2}
=
\frac{\alpha^2}{d^2}\frac{1}{(x+\alpha)^2}
+
\frac{\beta^2}{d^2}\frac{1}{(x+\beta)^2}
+
\frac{2\alpha\beta}{d^3}
\left(
\frac{1}{x+\beta}
-
\frac{1}{x+\alpha}
\right).
\end{aligned}
\]
Therefore
\[
\int_0^\infty\frac{dx}{(x+\alpha)^2}
=
\frac1\alpha,
\qquad
\int_0^\infty\frac{dx}{(x+\beta)^2}
=
\frac1\beta,
\]
while
\[
\begin{aligned}
\int_0^\infty
\left(
\frac{1}{x+\beta}
-
\frac{1}{x+\alpha}
\right)\,dx
=
\left[
\log(x+\beta)-\log(x+\alpha)
\right]_{0}^{\infty}
=
-\log\frac{\beta}{\alpha}.
\end{aligned}
\]
Substitution yields \eqref{eq:C-explicit}.

If $\alpha=\beta$, then
\[
\mathscr C(\alpha,\alpha)
=
\int_0^\infty
\frac{x^2}{(x+\alpha)^4}\,dx.
\]
Using $x=\alpha t$,
\[
\mathscr C(\alpha,\alpha)
=
\frac1\alpha
\int_0^\infty
\frac{t^2}{(1+t)^4}\,dt
=
\frac{1}{3\alpha},
\]
which proves \eqref{eq:C-diagonal}.

Finally, the change of variables $x=ct$ in
\eqref{eq:C-alpha-beta} gives
\[
\begin{aligned}
\mathscr C(c\alpha,c\beta)=
\int_0^\infty
\frac{x^2}
{(x+c\alpha)^2(x+c\beta)^2}
\,dx=
\frac1c
\int_0^\infty
\frac{t^2}
{(t+\alpha)^2(t+\beta)^2}
\,dt,
\end{aligned}
\]
which proves \eqref{eq:C-homogeneity}.
\end{proof}

The next observation will be useful in the parameter-recovery
argument.

\begin{lemma}
\label{lem:C-monotone}
For all \(\alpha,\beta>0\),
\[
\frac{\partial\mathscr C}{\partial\alpha}(\alpha,\beta)
=
-2\int_0^\infty
\frac{x^2}{(x+\alpha)^3(x+\beta)^2}\,dx<0,
\]
and
\[
\frac{\partial\mathscr C}{\partial\beta}(\alpha,\beta)
=
-2\int_0^\infty
\frac{x^2}{(x+\alpha)^2(x+\beta)^3}\,dx<0.
\]
Consequently, for fixed \(\beta>0\),
\(\alpha\mapsto\mathscr C(\alpha,\beta)\) is strictly decreasing on
\((0,\infty)\), and, for fixed \(\alpha>0\),
\(\beta\mapsto\mathscr C(\alpha,\beta)\) is strictly decreasing.
\end{lemma}

\begin{proof}
Fix \(\alpha,\beta>0\).  In a sufficiently small neighborhood of
\((\alpha,\beta)\), both parameters are bounded below by some
\(m>0\).  The absolute values of the partial derivatives of the
integrand in \eqref{eq:C-alpha-beta} are then bounded by
\[
g(x)=
\begin{cases}
2m^{-5}x^2, & 0<x\leq1,\\[1mm]
2x^{-3}, & x\geq1,
\end{cases}
\]
and \(g\in L^1(0,\infty)\).  Hence differentiation under the integral
sign is justified by dominated convergence, giving
\[
\frac{\partial\mathscr C}{\partial\alpha}(\alpha,\beta)
=
-2\int_0^\infty
\frac{x^2}{(x+\alpha)^3(x+\beta)^2}\,dx,
\]
and
\[
\frac{\partial\mathscr C}{\partial\beta}(\alpha,\beta)
=
-2\int_0^\infty
\frac{x^2}{(x+\alpha)^2(x+\beta)^3}\,dx.
\]
Since both integrands are strictly positive for \(x>0\), the two
partial derivatives are strictly negative.  The asserted strict
monotonicity follows immediately.
\end{proof}


\subsection{Separated parameters and ceiling errors}

For the sharp error estimate it is convenient to use the closed form
of $F$ obtained in \cite{LiuLuZu2026Quasi}.  We record only the case
needed for the asymptotic analysis.

Let
$1\leq u<\eta,
\tau=\eta-u$.
If $\tau\geq2$, using \cite[Sec.~3.3, Eq.~(3.18)]{LiuLuZu2026Quasi}, we have
\begin{equation}
\label{eq:F-separated}
F(u,\eta)
=
\frac{
\tau(u+\eta-1)
-
(2u\eta-u-\eta+1)
(H_{\eta-1}-H_{u-1})
}
{\tau(\tau^2-1)},
\end{equation}
where
\[
H_n
=
\sum_{m=1}^{n}\frac1m,
\qquad
H_0=0.
\]
Formula \eqref{eq:F-separated} will be the main discrete input in the large-parameter
analysis. We shall repeatedly use the standard asymptotic expansion
for the harmonic numbers (see, for example,
\cite[Sec.~5.4(ii), Eq.~(5.4.14), and
Sec.~5.11(i), Eq.~(5.11.2)]{NIST2010})
\begin{equation*}
\label{eq:harmonic-asymptotic}
H_n
=
\log n+\gamma+\frac{1}{2n}+O(n^{-2}),
\qquad n\to\infty,
\end{equation*}
where $\gamma$ denotes Euler's constant.

In the form occurring in
\eqref{eq:F-separated}, this becomes
$H_{n-1}
=
\log n+\gamma-\frac{1}{2n}
+O(n^{-2})$.
Consequently, whenever $u,\eta\to\infty$,
\begin{equation}
\label{eq:harmonic-difference}
H_{\eta-1}-H_{u-1}
=
\log\frac{\eta}{u}
+
\frac{1}{2u}
-
\frac{1}{2\eta}
+
O\left(u^{-2}+\eta^{-2}\right).
\end{equation}

In particular, if
\[
u=\alpha t+O(1),
\qquad
\eta=\beta t+O(1),
\qquad
0<\alpha<\beta,
\]
then
\begin{equation}
\label{eq:harmonic-scaled}
H_{\eta-1}-H_{u-1}
=
\log\frac{\beta}{\alpha}
+
O(t^{-1}).
\end{equation}

\begin{remark}
\label{rem:why-separated}
The diagonal and adjacent cases of \(F\), which involve second-order
harmonic sums and hence \(\pi^2\), are relevant to the exact
finite-index formulas in \cite{LiuLuZu2026Quasi}.  For the present asymptotic
problem, let
$a=\min\{k,\ell\},
b=\max\{k,\ell\}$,
and, for \(j\geq1\), set
\[
u_j=\left\lceil\frac{j}{b}\right\rceil,
\qquad
\eta_j=\left\lceil\frac{j}{a}\right\rceil.
\]
Since \(a<b\),
\[
\eta_j-u_j
\geq
\frac{b-a}{ab}\,j-1
\longrightarrow\infty.
\]
Hence \(\eta_j-u_j\geq2\) for all sufficiently large \(j\), so only
the separated regime \eqref{eq:F-separated} is needed for the
asymptotic tail.
\end{remark}


To apply the separated-parameter formula to the
quasi-homogeneous parameters, we now isolate the periodic part
of the ceiling functions.

For $m\in\mathbb N$ and $j\geq1$, define
\begin{equation*}
\label{eq:epsilon-m}
\varepsilon_m(j)
=
\left\lceil\frac jm\right\rceil-\frac jm.
\end{equation*}
Then
\begin{equation}
\label{eq:epsilon-range}
0\leq\varepsilon_m(j)<1,
\end{equation}
and
\begin{equation*}
\label{eq:ceiling-error}
\left\lceil\frac jm\right\rceil
=
\frac jm+\varepsilon_m(j).
\end{equation*}

If
$j=qm+r,
0\leq r\leq m-1$,
then
\begin{equation*}
\label{eq:epsilon-residue}
\varepsilon_m(j)
=
\begin{cases}
0,
& r=0,
\\[1mm]
1-\dfrac rm,
& 1\leq r\leq m-1.
\end{cases}
\end{equation*}

In particular,
$\varepsilon_m(j+m)=\varepsilon_m(j),
j\geq1$.
Thus the discrepancy between the discrete ceiling parameter and its
continuous approximation is a bounded periodic function.

Retain the notation
$a=\min\{k,\ell\},
b=\max\{k,\ell\}$,
introduced in Remark~\ref{rem:why-separated}.  Since
\[
\{A_j,B_j\}
=
\left\{
\left\lceil\frac{j}{a}\right\rceil,
\left\lceil\frac{j}{b}\right\rceil
\right\},
\]
we record the corresponding ordered ceiling parameters as
\begin{equation}
\label{eq:u-eta-j}
u_j
=
\left\lceil\frac jb\right\rceil,
\qquad
\eta_j
=
\left\lceil\frac ja\right\rceil.
\end{equation}
Then
\begin{equation}
\label{eq:u-eta-ceiling-errors}
u_j
=
\frac jb+\varepsilon_b(j),
\qquad
\eta_j
=
\frac ja+\varepsilon_a(j),
\end{equation}
and, by the symmetry of $F$ and \eqref{eq:Sigma-F-recalled}, we have
$\Sigma_j(M_{k,\ell})
=
F(u_j,\eta_j)$.

Consequently,
\begin{equation}
\label{eq:parameter-separation}
\eta_j-u_j
=
\frac{b-a}{ab}\,j
+
\varepsilon_a(j)-\varepsilon_b(j)
=
\frac{b-a}{ab}\,j+O(1).
\end{equation}
This is consistent with the separation established in
Remark~\ref{rem:why-separated}.

Since the ceiling errors are bounded,
\eqref{eq:u-eta-ceiling-errors} gives
\[
u_j=\frac1b\,j+O(1),
\qquad
\eta_j=\frac1a\,j+O(1).
\]
Since \(a<b\), equation \eqref{eq:harmonic-scaled}, with
\(t=j\), \(\alpha=1/b\), and \(\beta=1/a\), yields
\[
H_{\eta_j-1}-H_{u_j-1}
=
\log\frac ba+O(j^{-1}).
\]

Since $\varepsilon_a$ and $\varepsilon_b$ are periodic with periods
$a$ and $b$, respectively, any quantity whose dependence on $j$
occurs only through $\varepsilon_a(j)$ and $\varepsilon_b(j)$ is
$L$-periodic, where
$L:=\operatorname{lcm}(a,b)=\operatorname{lcm}(k,\ell)$.
This periodicity will account for the arithmetic fluctuations beyond
the leading $j^{-1}$ term.



\section{Leading asymptotics and asymptotic rigidity}
\label{sec3}

We now apply the exact representation and the discrete asymptotic
preliminaries of Section~\ref{sec2} to determine the first-order large-index
behavior of $\Sigma_j(M_{k,\ell})$.

\subsection{A bounded-perturbation scaling law for \texorpdfstring{\(F\)}{F}}

The following proposition isolates the asymptotic mechanism that will
be applied to the ceiling parameters.

\begin{proposition}
\label{prop:F-scaled-asymptotic}
Let $0<\alpha<\beta$ and let $M>0$.  Suppose that
$\{u_n\}$ and $\{\eta_n\}$ are sequences of positive integers satisfying
\begin{equation}
\label{eq:bounded-perturbation}
|u_n-\alpha n|\leq M,
\qquad
|\eta_n-\beta n|\leq M
\end{equation}
for all sufficiently large $n$.  Then
\begin{equation}
\label{eq:F-scaled-asymptotic}
F(u_n,\eta_n)
=
\frac{\mathscr C(\alpha,\beta)}{n}
+
O(n^{-2}),
\qquad n\to\infty,
\end{equation}
where the implied constant may depend on
$\alpha,\beta$ and $M$.

In particular,
\begin{equation}
\label{eq:F-scaled-limit}
\lim_{n\to\infty}
nF(u_n,\eta_n)
=
\mathscr C(\alpha,\beta).
\end{equation}
\end{proposition}

\begin{proof}
Write
$u_n=\alpha n+\rho_n,
\eta_n=\beta n+\sigma_n$, by \eqref{eq:bounded-perturbation},
$|\rho_n|\leq M,
|\sigma_n|\leq M$
for all sufficiently large $n$.  Set
$d=\beta-\alpha>0$
and
$\tau_n=\eta_n-u_n$.
Then
\begin{equation}
\label{eq:tau-n-expansion}
\tau_n
=
dn+O(1).
\end{equation}
In particular, $\tau_n\to\infty$, and hence
$1\leq u_n<\eta_n,
\tau_n\geq2$
for all sufficiently large $n$.

We may therefore use \eqref{eq:F-separated}:
\begin{equation}
\label{eq:F-separated-scaled}
F(u_n,\eta_n)
=
\frac{
\tau_n(u_n+\eta_n-1)
-
(2u_n\eta_n-u_n-\eta_n+1)
(H_{\eta_n-1}-H_{u_n-1})
}
{\tau_n(\tau_n^2-1)}.
\end{equation}


Since \eqref{eq:bounded-perturbation} gives
$u_n=\alpha n+O(1),
\eta_n=\beta n+O(1)$,
equation \eqref{eq:harmonic-scaled}, with \(t=n\), yields
\begin{equation}
\label{eq:harmonic-scaled-precise}
H_{\eta_n-1}-H_{u_n-1}
=
\log\frac{\beta}{\alpha}
+
O(n^{-1}).
\end{equation}

Next,
\begin{equation*}
\label{eq:sum-expansion}
u_n+\eta_n-1
=
(\alpha+\beta)n+O(1),
\end{equation*}
and hence, by \eqref{eq:tau-n-expansion},
\begin{equation*}
\label{eq:first-numerator-expansion}
\tau_n(u_n+\eta_n-1)
=
d(\alpha+\beta)n^2+O(n).
\end{equation*}
Moreover,
\begin{equation}
\label{eq:product-expansion}
2u_n\eta_n-u_n-\eta_n+1
=
2\alpha\beta n^2+O(n).
\end{equation}

Combining \eqref{eq:harmonic-scaled-precise} and
\eqref{eq:product-expansion}, we obtain
\begin{equation*}
\label{eq:second-numerator-expansion}
\begin{aligned}
(2u_n\eta_n-u_n-\eta_n+1)
(H_{\eta_n-1}-H_{u_n-1})=
2\alpha\beta
\log\frac{\beta}{\alpha}\,n^2
+
O(n).
\end{aligned}
\end{equation*}
Thus the numerator in \eqref{eq:F-separated-scaled} equals
\begin{equation}
\label{eq:numerator-asymptotic}
\left[
d(\alpha+\beta)
-
2\alpha\beta\log\frac{\beta}{\alpha}
\right]n^2
+
O(n).
\end{equation}

For the denominator, \eqref{eq:tau-n-expansion} gives
$\tau_n^3
=
d^3n^3+O(n^2)$,
and therefore
\begin{equation*}
\label{eq:denominator-asymptotic}
\tau_n(\tau_n^2-1)
=
d^3n^3+O(n^2).
\end{equation*}
Equivalently,
\[
\tau_n(\tau_n^2-1)
=
d^3n^3\left(1+O(n^{-1})\right),
\]
and hence
\[
\frac{1}{\tau_n(\tau_n^2-1)}
=
\frac{1}{d^3n^3}
\left(1+O(n^{-1})\right).
\]

Using this together with
\eqref{eq:numerator-asymptotic}, we obtain
\[
\begin{aligned}
F(u_n,\eta_n)
&=
\frac1n
\left[
\frac{\alpha+\beta}{d^2}
-
\frac{2\alpha\beta}{d^3}
\log\frac{\beta}{\alpha}
\right]
+
O(n^{-2}).
\end{aligned}
\]
Since $d=\beta-\alpha$, Lemma~\ref{lem:C-explicit} shows that
\[
\frac{\alpha+\beta}{d^2}
-
\frac{2\alpha\beta}{d^3}
\log\frac{\beta}{\alpha}
=
\mathscr C(\alpha,\beta).
\]
This proves \eqref{eq:F-scaled-asymptotic}.  Multiplication by $n$
and passage to the limit yield \eqref{eq:F-scaled-limit}.
\end{proof}

\subsection{The invariant tail and its consequences}

We now specialize the preceding scaling law to the numerical
invariant sequence of \(M_{k,\ell}\).  Retain the notation $a=\min\{k,\ell\}$ and $b=\max\{k,\ell\}$.
Since $k\neq\ell$,
we have $1\leq a<b$.
Recall from \eqref{eq:u-eta-j} that
\[
u_j
=
\left\lceil\frac{j}{b}\right\rceil,
\qquad
\eta_j
=
\left\lceil\frac{j}{a}\right\rceil.
\]
By \eqref{eq:u-eta-ceiling-errors},
\begin{equation*}
\label{eq:u-eta-linear-tail}
u_j
=
\frac{j}{b}+O(1),
\qquad
\eta_j
=
\frac{j}{a}+O(1).
\end{equation*}
Since \(\{u_j,\eta_j\}=\{A_j,B_j\}\), it follows from \eqref{eq:Sigma-F-recalled} and the symmetry of \(F\) that
$\Sigma_j(M_{k,\ell})
=
F(u_j,\eta_j)$.

\begin{theorem}
\label{thm:tail-main}
Let
\[
M_{k,\ell}
=
[z^k-w^\ell],
\qquad
k,\ell\in\mathbb N,
\qquad
k\neq\ell.
\]
Then
\begin{equation}
\label{eq:Sigma-tail}
\Sigma_j(M_{k,\ell})
=
\frac{C_{k,\ell}}{j}
+
O(j^{-2}),
\qquad
j\to\infty,
\end{equation}
where
\begin{equation}
\label{eq:Ckl-definition}
C_{k,\ell}
:=
\mathscr C\left(\frac1k,\frac1\ell\right)
=
\int_0^\infty
\frac{x^2}
{\left(x+\frac1k\right)^2
 \left(x+\frac1\ell\right)^2}
\,dx.
\end{equation}
Consequently, we have
\begin{equation}
\label{eq:jSigma-limit}
\lim_{j\to\infty}
j\,\Sigma_j(M_{k,\ell})
=
C_{k,\ell}>0,
\end{equation}
and
\begin{equation}
\label{eq:Ckl-explicit}
C_{k,\ell}
=
\frac{ab(a+b)}{(b-a)^2}
-
\frac{2a^2b^2}{(b-a)^3}
\log\frac{b}{a}.
\end{equation}
\end{theorem}

\begin{proof}
By \eqref{eq:u-eta-ceiling-errors},
\[
u_j
=
\frac1b\,j+O(1),
\qquad
\eta_j
=
\frac1a\,j+O(1).
\]

Since
\[
0<\frac1b<\frac1a,
\]
Proposition~\ref{prop:F-scaled-asymptotic}, with
$\alpha=\frac1b,
\beta=\frac1a$,
gives
\[
\Sigma_j(M_{k,\ell})
=
F(u_j,\eta_j)
=
\frac1j
\mathscr C\left(\frac1b,\frac1a\right)
+
O(j^{-2}).
\]
The symmetry of $\mathscr C$ implies
\[
\mathscr C\left(\frac1b,\frac1a\right)
=
\mathscr C\left(\frac1k,\frac1\ell\right),
\]
and hence \eqref{eq:Sigma-tail} and
\eqref{eq:Ckl-definition} follow.

Since the integrand in \eqref{eq:Ckl-definition} is strictly positive
for every $x>0$, we have $C_{k,\ell}>0$.
Multiplying \eqref{eq:Sigma-tail} by $j$ therefore gives
\eqref{eq:jSigma-limit}.

It remains to evaluate the coefficient explicitly.  By~\eqref{eq:C-explicit},
\[
\begin{aligned}
C_{k,\ell}=
\mathscr C\left(\frac1b,\frac1a\right)=
\frac{ab(a+b)}{(b-a)^2}
-
\frac{2a^2b^2}{(b-a)^3}
\log\frac ba,
\end{aligned}
\]
which proves \eqref{eq:Ckl-explicit}.
\end{proof}

\begin{remark}
By symmetry,
$C_{k,\ell}=C_{\ell,k}$.
The bounded periodic ceiling errors do not affect the leading
coefficient; their residue-class dependence first appears in the
lower-order corrections studied in Section~\ref{sec4}.
\end{remark}

Define
\begin{equation}
\label{eq:renormalized-remainder}
r_j^{(k,\ell)}
=
\Sigma_j(M_{k,\ell})
-
\frac{C_{k,\ell}}{j},
\qquad
j\geq1.
\end{equation}
This renormalized remainder will be used later to isolate the leading logarithmic singularity of Yang's generating function.
By Theorem~\ref{thm:tail-main},
\begin{equation}
\label{eq:renormalized-remainder-bound}
r_j^{(k,\ell)}
=
O(j^{-2}).
\end{equation}

\begin{corollary}
\label{cor:renormalized-tail}
The series
\begin{equation}
\label{eq:renormalized-tail-series}
\sum_{j=1}^{\infty}
\left(
\Sigma_j(M_{k,\ell})
-
\frac{C_{k,\ell}}{j}
\right)
\end{equation}
converges absolutely.  If
\begin{equation*}
\label{eq:Rkl-definition}
\mathcal R_{k,\ell}
:=
\sum_{j=1}^{\infty}
\left(
\Sigma_j(M_{k,\ell})
-
\frac{C_{k,\ell}}{j}
\right),
\end{equation*}
then
\begin{equation}
\label{eq:partial-sum-asymptotic}
\sum_{j=0}^{N}\Sigma_j(M_{k,\ell})
=
C_{k,\ell}\log N
+
S_{k,\ell}
+
O(N^{-1}),
\qquad
N\to\infty,
\end{equation}
where
$S_{k,\ell}
=
\Sigma_0(M_{k,\ell})
+
C_{k,\ell}\gamma
+
\mathcal R_{k,\ell}$.
\end{corollary}

\begin{proof}
Absolute convergence of \eqref{eq:renormalized-tail-series} follows
from \eqref{eq:renormalized-remainder-bound} and
$\sum_{j=1}^{\infty}j^{-2}<\infty$.

Moreover,
\[
\begin{aligned}
\sum_{j=0}^{N}\Sigma_j(M_{k,\ell})
&=
\Sigma_0(M_{k,\ell})
+
C_{k,\ell}H_N
+
\sum_{j=1}^{N}r_j^{(k,\ell)}.
\end{aligned}
\]

Since
$H_N
=
\log N+\gamma+O(N^{-1})$
and
\[
\begin{aligned}
\mathcal R_{k,\ell}
-
\sum_{j=1}^{N}r_j^{(k,\ell)}
=
\sum_{j=N+1}^{\infty}r_j^{(k,\ell)}
=
O\left(
\sum_{j=N+1}^{\infty}j^{-2}
\right)
=
O(N^{-1}),
\end{aligned}
\]
we obtain
\[
\sum_{j=0}^{N}\Sigma_j(M_{k,\ell})
=
C_{k,\ell}\log N
+
\Sigma_0(M_{k,\ell})
+
C_{k,\ell}\gamma
+
\mathcal R_{k,\ell}
+
O(N^{-1}),
\]
which proves \eqref{eq:partial-sum-asymptotic}.
\end{proof}




The first-order analysis above shows that the tail of the numerical
invariant sequence determines the coefficient
$C_{k,\ell}
=
\lim_{j\to\infty} j\Sigma_j(M_{k,\ell})$.
We now examine the parameter information contained in this coefficient.

We shall use one result from our preceding work
\cite[Sec.~4.1]{LiuLuZu2026Quasi}.  The strict descent set was shown there to be
\begin{equation}
\label{eq:descent-set-recalled}
\mathcal D_{k,\ell}
=
\left\{
j\in\mathbb N:
\Sigma_j(M_{k,\ell})
>
\Sigma_{j+1}(M_{k,\ell})
\right\}
=
k\mathbb N\cup\ell\mathbb N.
\end{equation}
Consequently,
\begin{equation}
\label{eq:first-descent-recalled}
\min\mathcal D_{k,\ell}
=
\min\{k,\ell\}.
\end{equation}
These facts are recalled only as input. Our next goal is to show that
the single discrete quantity in the preceding identity, together with
the leading tail coefficient obtained above, already determines the
unordered pair of exponents.

\subsection{The normalized tail profile}

Retain the notation $a=\min\{k,\ell\}$ and $b=\max\{k,\ell\}$. Let
$\rho=\frac ba$.
Since $k\neq\ell$, we have
$1\leq a<b$ and
$\rho>1$.

For $\rho>1$, define
\begin{equation}
\label{eq:Phi-definition}
\Phi(\rho)
=
\mathscr C\left(\frac1\rho,1\right).
\end{equation}
By Lemma~\ref{lem:C-explicit},
\begin{equation*}
\label{eq:Phi-explicit}
\Phi(\rho)
=
\frac{\rho(\rho+1)}{(\rho-1)^2}
-
\frac{2\rho^2}{(\rho-1)^3}
\log\rho,
\qquad
\rho>1.
\end{equation*}

\begin{lemma}
\label{lem:Phi-profile}
The function
\[
\Phi:(1,\infty)\longrightarrow
\left(\frac13,1\right)
\]
is a continuous strictly increasing bijection.  More precisely,
\begin{equation}
\label{eq:Phi-endpoints}
\lim_{\rho\to1^+}\Phi(\rho)
=
\frac13,
\qquad
\lim_{\rho\to\infty}\Phi(\rho)
=
1.
\end{equation}
\end{lemma}

\begin{proof}
 From \eqref{eq:Phi-definition},
\begin{equation}
\label{eq:Phi-integral}
\Phi(\rho)
=
\int_0^\infty
\frac{x^2}
{\left(x+\frac1\rho\right)^2(x+1)^2}
\,dx.
\end{equation}

By Lemma~\ref{lem:C-monotone}, the function
\(\alpha\mapsto\mathscr C(\alpha,1)\) is strictly decreasing.
Since \(\rho\mapsto1/\rho\) is strictly decreasing on
\((1,\infty)\), \eqref{eq:Phi-definition} shows that
\(\Phi\) is strictly increasing.

From \eqref{eq:Phi-integral}, for every \(\rho>1\) and \(x>0\),
\[
0\leq
\frac{x^2}
{\left(x+\frac1\rho\right)^2(x+1)^2}
\leq
\frac1{(x+1)^2},
\]
and \((x+1)^{-2}\in L^1(0,\infty)\).
If \(\rho_n\to\rho_0>1\), then the integrand converges pointwise to
\[
\frac{x^2}
{\left(x+\frac1{\rho_0}\right)^2(x+1)^2}.
\]
The dominated convergence theorem therefore gives
\(\Phi(\rho_n)\to\Phi(\rho_0)\), proving that \(\Phi\) is continuous
on \((1,\infty)\).

As $\rho\to1^+$, the same domination gives
\[
\lim_{\rho\to1^+}\Phi(\rho)
=
\int_0^\infty\frac{x^2}{(x+1)^4}\,dx
=
\mathscr C(1,1)
=
\frac13,
\]
where the last identity follows from \eqref{eq:C-diagonal}.
Likewise, as \(\rho\to\infty\),
\[
\frac{x^2}
{\left(x+\frac1\rho\right)^2(x+1)^2}
\longrightarrow
\frac1{(x+1)^2}
\qquad (x>0),
\]
and hence
\[
\lim_{\rho\to\infty}\Phi(\rho)
=
\int_0^\infty\frac{dx}{(x+1)^2}
=
1.
\]
Thus \eqref{eq:Phi-endpoints} holds.

Since \(\Phi\) is continuous and strictly increasing and has endpoint
limits \(1/3\) and \(1\), its range is exactly
\((1/3,1)\). Hence
$\Phi:(1,\infty)\longrightarrow(1/3,1)$
is a bijection.
\end{proof}

\subsection{Scale and shape of the leading tail coefficient}

The preceding profile separates the absolute size of the two
exponents from their relative ratio.

\begin{proposition}
\label{prop:C-scale-shape}
Let
$a=\min\{k,\ell\},
b=\max\{k,\ell\}$.
Then
\begin{equation}
\label{eq:C-scale-shape}
C_{k,\ell}
=
a\,
\Phi\left(\frac ba\right).
\end{equation}
In particular,
\begin{equation}
\label{eq:C-normalized}
\frac{C_{k,\ell}}{a}
=
\Phi\left(\frac ba\right),
\end{equation}
and
\begin{equation}
\label{eq:C-bounds}
\frac a3
<
C_{k,\ell}
<
a.
\end{equation}
For each fixed $a\in\mathbb N$, the quantity $C_{a,b}$ is strictly
increasing as a function of the integer parameter $b>a$.
\end{proposition}




\begin{proof}
Since \(\mathscr C\) is symmetric and
\(\{a,b\}=\{k,\ell\}\),
\[
C_{k,\ell}
=
\mathscr C\left(\frac1b,\frac1a\right).
\]
Using the homogeneity relation \eqref{eq:C-homogeneity}
$\mathscr C(c\alpha,c\beta)
=
c^{-1}\mathscr C(\alpha,\beta)$,
with \(c=1/a\), we obtain
\[
C_{k,\ell}
=
a\,\mathscr C\left(\frac ab,1\right)
=
a\,\Phi\left(\frac ba\right),
\]
which proves \eqref{eq:C-scale-shape}. Dividing by \(a\) gives
\eqref{eq:C-normalized}.

Since \(b/a>1\), Lemma~\ref{lem:Phi-profile} yields
\[
\frac13
<
\Phi\left(\frac ba\right)
<
1.
\]
Multiplying by \(a>0\) gives
\eqref{eq:C-bounds}.

Finally, for fixed \(a\), the map
\(b\mapsto b/a\) is strictly increasing, and
\(\Phi\) is strictly increasing by
Lemma~\ref{lem:Phi-profile}. Hence
$C_{a,b}
=
a\,\Phi\left(\frac ba\right)$
is strictly increasing in \(b>a\).
\end{proof}

\begin{remark}
Formula~\eqref{eq:C-scale-shape} separates scale from ratio: the factor $a$ determines
the scale, whereas
$\frac{C_{k,\ell}}{a}$
depends only on $b/a$. Consistently, the homogeneity of $C$ gives
$C_{dk,d\ell}=dC_{k,\ell}, d\in\mathbb N$,
so simultaneous scaling of the two exponents changes the leading
tail coefficient by the same factor while leaving their ratio
unchanged.
\end{remark}

\subsection{Recovery from coarse asymptotic data}

We now combine the first strict descent with the leading tail coefficient.

\begin{theorem}[Asymptotic rigidity]
\label{thm:asymptotic-rigidity}
Let
\[
M_{k,\ell}=[z^k-w^\ell],\qquad k,\ell\in\mathbb N,\quad k\ne\ell.
\]
Recall that
$a=\min\mathcal D_{k,\ell}=\min\{k,\ell\}$.
Then
\[
\frac13<\frac{C_{k,\ell}}{a}<1.
\]
Moreover, the unordered pair $\{k,\ell\}$ is uniquely determined by
\[
\left(
\min\mathcal D_{k,\ell},
\lim_{j\to\infty}j\Sigma_j(M_{k,\ell})
\right).
\]
More precisely,
\begin{equation}
\label{eq:asymptotic-recovery}
\{k,\ell\}
=
\left\{
a,\,
a\Phi^{-1}\left(\frac{C_{k,\ell}}{a}\right)
\right\}.
\end{equation}

Equivalently, if
$k',\ell'\in\mathbb N,
k'\neq\ell'$,
satisfy
$\min\mathcal D_{k,\ell}
=
\min\mathcal D_{k',\ell'}$
and
\[
\lim_{j\to\infty}
j\,\Sigma_j(M_{k,\ell})
=
\lim_{j\to\infty}
j\,\Sigma_j(M_{k',\ell'}),
\]
then
$\{k,\ell\}
=
\{k',\ell'\}$.
\end{theorem}

\begin{proof}
By \eqref{eq:first-descent-recalled}, we have
$a
=
\min\mathcal D_{k,\ell}$.
Then
$b>a$.
By Theorem~\ref{thm:tail-main},
$\lim_{j\to\infty}j\Sigma_j(M_{k,\ell})=C_{k,\ell}$.

Hence Proposition~\ref{prop:C-scale-shape} gives
\[
\frac{C_{k,\ell}}{a}
=
\Phi\left(\frac{b}{a}\right).
\]

Since
$\frac ba>1$,
Lemma~\ref{lem:Phi-profile} implies
\[
\frac13
<
\frac{C_{k,\ell}}{a}
<
1.
\]

The function
$\Phi:(1,\infty)
\longrightarrow
\left(\frac13,1\right)$
is bijective, so
$\frac ba
=
\Phi^{-1}
\left(
\frac{C_{k,\ell}}{a}
\right)$.
Therefore
\[
b
=
a
\Phi^{-1}
\left(
\frac{C_{k,\ell}}{a}
\right),
\]
which proves \eqref{eq:asymptotic-recovery}.

For the equivalent injectivity statement, the two assumed equalities
give the same values of $a$ and $C_{k,\ell}$.  Formula
\eqref{eq:asymptotic-recovery} then gives the same unordered pair of
exponents.
\end{proof}

The gain over the exact recovery mechanism in
\cite{LiuLuZu2026Quasi} is particularly transparent in the
divisibility case. If
$\{k,\ell\}=\{a,qa\}, q\geq2$,
then \(\mathcal D_{k,\ell}=a\mathbb N\), so the descent set alone does not
determine \(q\). In the exact recovery argument, the larger exponent
is distinguished by the values of the constant blocks, whereas here,
once \(a\) is known, the single leading tail coefficient \(C_{k,\ell}\) is
sufficient.

\begin{remark}
The recovery in Theorem~\ref{thm:asymptotic-rigidity} is necessarily unordered, since both
$\min\mathcal D_{k,\ell}$ and $C_{k,\ell}$ are invariant under
$(k,\ell)\longleftrightarrow(\ell,k)$.
Moreover, Theorem~\ref{thm:asymptotic-rigidity} is a sufficiency statement: it does not assert
that either datum alone is sufficient, or that the pair of data is
minimal. 
\end{remark}


\section{Second-order arithmetic fluctuations and descent amplitudes}
\label{sec4}

The leading tail coefficient $C_{k,\ell}$ does not see the bounded
ceiling errors.  We now pass to the next asymptotic scale, where the
arithmetic staircase becomes visible.  Set
\[
\alpha=\frac1b,
\qquad
\beta=\frac1a.
\]
Then $0<\alpha<\beta$.

We continue to use the separated-parameter formula~\eqref{eq:F-separated} for $F$.  The aim here is not to obtain another exact
formula, but to extract its second-order behavior when both arguments
grow linearly.

\subsection{A two-term scaling expansion}

For $0<\alpha<\beta$, define
\begin{equation}\label{eq:Q-def}
\mathscr Q(\alpha,\beta)
=
\frac{(\alpha+\beta)\log(\beta/\alpha)-2(\beta-\alpha)}
{(\beta-\alpha)^3}.
\end{equation}

We first establish a two-term scaling expansion for general
integer parameter sequences; the ceiling parameters
\eqref{eq:u-eta-j} will be substituted into this result
in the next subsection.

\begin{proposition}
\label{prop:F-two-term-scaling}
Let $\{u_n\}$ and $\{\eta_n\}$ be sequences of positive integers such
that
\[
\frac{u_n}{n}\longrightarrow\alpha,
\qquad
\frac{\eta_n}{n}\longrightarrow\beta,
\qquad
0<\alpha<\beta.
\]
Set
$\alpha_n=\frac{u_n}{n},
\beta_n=\frac{\eta_n}{n}$.
Then
\begin{equation}\label{eq:F-two-term}
F(u_n,\eta_n)
=
\frac{\mathscr C(\alpha_n,\beta_n)}{n}
+
\frac{\mathscr Q(\alpha_n,\beta_n)}{n^2}
+
O(n^{-3}).
\end{equation}
The estimate is uniform whenever
$(\alpha_n,\beta_n)$ remains in a compact subset of
$\{(x,y)\in(0,\infty)^2:x<y\}$.
\end{proposition}

\begin{proof}
Put
\[
d_n=\beta_n-\alpha_n,
\qquad
L_n=\log\frac{\beta_n}{\alpha_n}.
\]
Since
$\alpha_n\longrightarrow\alpha,
\beta_n\longrightarrow\beta,
\alpha<\beta$,
the quantities $\alpha_n$ and $\beta_n$ are bounded away from zero,
while $d_n$ is bounded away from zero, for all sufficiently large
$n$.  In particular,
$\eta_n-u_n
=
nd_n\geq2$
eventually, and hence the separated-parameter formula ~\eqref{eq:F-separated} applies.

The harmonic-number expansion \eqref{eq:harmonic-difference} from Section~\ref{sec2} gives
\[
\begin{aligned}
H_{\eta_n-1}-H_{u_n-1}
=
\log\frac{\eta_n}{u_n}
+
\frac1{2u_n}
-
\frac1{2\eta_n}
+
O(n^{-2})
=
L_n
+
\frac{d_n}{2n\alpha_n\beta_n}
+
O(n^{-2}).
\end{aligned}
\]

Since
$u_n=n\alpha_n,
\eta_n=n\beta_n$,
we have
\[
\begin{aligned}
(\eta_n-u_n)(u_n+\eta_n-1)
=
nd_n
\bigl(
n(\alpha_n+\beta_n)-1
\bigr)=
n^2d_n(\alpha_n+\beta_n)-nd_n
\end{aligned}
\]
and
\[
2u_n\eta_n-u_n-\eta_n+1
=
2n^2\alpha_n\beta_n
-
n(\alpha_n+\beta_n)
+
1.
\]
Therefore
\[
\begin{aligned}
\bigl(
2u_n\eta_n-u_n-\eta_n+1
\bigr)
\bigl(
H_{\eta_n-1}-H_{u_n-1}
\bigr)=
2n^2\alpha_n\beta_nL_n
+
n\bigl[
d_n-(\alpha_n+\beta_n)L_n
\bigr]
+
O(1).
\end{aligned}
\]
It follows that the numerator of the separated formula equals
\[
\begin{aligned}
n^2
\bigl[
d_n(\alpha_n+\beta_n)
-
2\alpha_n\beta_nL_n
\bigr]+
n
\bigl[
(\alpha_n+\beta_n)L_n-2d_n
\bigr]
+
O(1).
\end{aligned}
\]

The denominator is
\[
\begin{aligned}
(\eta_n-u_n)
\bigl(
(\eta_n-u_n)^2-1
\bigr)=
nd_n(n^2d_n^2-1)=
n^3d_n^3-nd_n.
\end{aligned}
\]
Since $d_n$ is bounded away from zero,
\[
\frac1{n^3d_n^3-nd_n}
=
\frac1{n^3d_n^3}
+
O(n^{-5}).
\]
Combining the preceding estimates with \eqref{eq:F-separated} yields
\[
\begin{aligned}
F(u_n,\eta_n)
=
\frac1n
\left[
\frac{\alpha_n+\beta_n}{d_n^2}
-
\frac{2\alpha_n\beta_nL_n}{d_n^3}
\right]+
\frac1{n^2}
\left[
\frac{
(\alpha_n+\beta_n)L_n-2d_n
}{d_n^3}
\right]
+
O(n^{-3}).
\end{aligned}
\]
The first bracket is \(\mathscr C(\alpha_n,\beta_n)\), while the
second is \(\mathscr Q(\alpha_n,\beta_n)\) by \eqref{eq:Q-def}.

Finally, suppose that \((\alpha_n,\beta_n)\) remains in a compact
subset \(K\) of
$\{(x,y)\in(0,\infty)^2:x<y\}$.
Then, for some \(m,M,\delta>0\) depending only on \(K\),
\[
m\leq\alpha_n,\beta_n\leq M,
\qquad
d_n=\beta_n-\alpha_n\geq\delta.
\]
Thus \(L_n=\log(\beta_n/\alpha_n)\) is uniformly bounded, while
\(\alpha_n\), \(\beta_n\), and \(d_n\) stay uniformly away from the
singular sets occurring in the preceding expansions. Consequently,
the harmonic-number remainder, the reciprocal-denominator expansion,
and hence the final \(O(n^{-3})\) remainder are uniform on \(K\).
This proves the proposition.
\end{proof}


\subsection{The periodic second-order coefficient}

Recall
\[
\varepsilon_m(j)
=
\left\lceil\frac jm\right\rceil-\frac jm.
\]
Then
\[
u_j
=
\left\lceil\frac jb\right\rceil
=
\frac jb+\varepsilon_b(j),
\qquad
\eta_j
=
\left\lceil\frac ja\right\rceil
=
\frac ja+\varepsilon_a(j).
\]
Hence
\[
\frac{u_j}{j}
=
\alpha+\frac{\varepsilon_b(j)}j,
\qquad
\frac{\eta_j}{j}
=
\beta+\frac{\varepsilon_a(j)}j.
\]

Write
\[
\mathscr C_1(\alpha,\beta)
=
\frac{\partial\mathscr C}{\partial\alpha}(\alpha,\beta),
\qquad
\mathscr C_2(\alpha,\beta)
=
\frac{\partial\mathscr C}{\partial\beta}(\alpha,\beta).
\]
By Lemma~\ref{lem:C-monotone},
$\mathscr C_1(\alpha,\beta)<0,
\mathscr C_2(\alpha,\beta)<0$.

Define
\begin{equation}\label{eq:Psi-second-order}
\Psi_{k,\ell}(j)
=
\mathscr Q(\alpha,\beta)
+
\varepsilon_b(j)\mathscr C_1(\alpha,\beta)
+
\varepsilon_a(j)\mathscr C_2(\alpha,\beta).
\end{equation}

\begin{theorem}
\label{thm:second-order-periodic}
Let
$M_{k,\ell}=[z^k-w^\ell],
k\neq\ell$.
Then
\begin{equation}\label{eq:Sigma-two-term}
\Sigma_j(M_{k,\ell})
=
\frac{C_{k,\ell}}{j}
+
\frac{\Psi_{k,\ell}(j)}{j^2}
+
O(j^{-3}),
\qquad j\to\infty.
\end{equation}
Moreover, $\Psi_{k,\ell}$ is periodic, and
$\operatorname{lcm}(k,\ell)$ is a period of $\Psi_{k,\ell}$.
\end{theorem}

\begin{proof}
Applying \eqref{eq:F-two-term} with
\[
n=j,\qquad u_n=u_j,\qquad \eta_n=\eta_j,
\]
we obtain
\[
\Sigma_j(M_{k,\ell})
=
\frac{\mathscr C(\alpha_j,\beta_j)}j
+
\frac{\mathscr Q(\alpha_j,\beta_j)}{j^2}
+
O(j^{-3}).
\]

The ceiling errors are bounded, and $\mathscr C$ is smooth in a
neighborhood of $(\alpha,\beta)$.  Taylor's formula therefore gives
\[
\begin{aligned}
\mathscr C(\alpha_j,\beta_j)=
\mathscr C(\alpha,\beta)+
\frac{
\varepsilon_b(j)\mathscr C_1(\alpha,\beta)
+
\varepsilon_a(j)\mathscr C_2(\alpha,\beta)
}{j}
+
O(j^{-2}).
\end{aligned}
\]
Similarly, we have $\mathscr Q(\alpha_j,\beta_j)
=
\mathscr Q(\alpha,\beta)
+
O(j^{-1})$.
Substitution yields
\[
\Sigma_j(M_{k,\ell})
=
\frac{\mathscr C(\alpha,\beta)}j
+
\frac{\Psi_{k,\ell}(j)}{j^2}
+
O(j^{-3}).
\]
Since
$\mathscr C(\alpha,\beta)
=
C_{k,\ell}$,
the asserted expansion follows.

Finally, $\varepsilon_a$ and $\varepsilon_b$ have periods $a$ and
$b$, respectively.  Hence their linear combination
$\Psi_{k,\ell}$ has
$\operatorname{lcm}(a,b)
=
\operatorname{lcm}(k,\ell)$
as a period.
\end{proof}

\subsection{Leading descent amplitudes and their mean}

The exact locations of the strict decreases were already determined
in \cite{LiuLuZu2026Quasi}.  We do not reprove that result here.
Instead, we determine the leading size of the decreases at those
positions.



Define
\[
\Lambda_b
=
-\mathscr C_1\left(\frac1b,\frac1a\right),
\qquad
\Lambda_a
=
-\mathscr C_2\left(\frac1b,\frac1a\right).
\]
By Lemma~\ref{lem:C-monotone},
\[
\Lambda_b
=
2\int_0^\infty
\frac{x^2}
{\left(x+\frac1b\right)^3
 \left(x+\frac1a\right)^2}
\,dx>0,
\qquad
\Lambda_a
=
2\int_0^\infty
\frac{x^2}
{\left(x+\frac1b\right)^2
 \left(x+\frac1a\right)^3}
\,dx>0.
\]

For a statement $E$, let $\mathbf 1_E$ denote its indicator. The ceiling errors satisfy
\begin{equation}\label{eq:ceiling-increment}
\varepsilon_m(j+1)-\varepsilon_m(j)
=
-\frac{1}{m}
+
\mathbf 1_{\{m\mid j\}}.
\end{equation}
Also, differentiating the homogeneity relation
$\mathscr C(c\alpha,c\beta)
=
c^{-1}\mathscr C(\alpha,\beta)$
with respect to $c$ at $c=1$ gives
\begin{equation}\label{eq:Euler-C}
\alpha\mathscr C_1(\alpha,\beta)
+
\beta\mathscr C_2(\alpha,\beta)
=
-\mathscr C(\alpha,\beta).
\end{equation}

At
$\alpha=\frac1b,
\beta=\frac1a$,
identity \eqref{eq:Euler-C} becomes
\[
C_{k,\ell}
=
\frac{\Lambda_b}{b}
+
\frac{\Lambda_a}{a}.
\]

\begin{theorem}
\label{thm:descent-amplitudes}
The successive differences satisfy
\begin{equation}\label{eq:drop-asymptotic}
\Sigma_j(M_{k,\ell})-\Sigma_{j+1}(M_{k,\ell})
=
\frac{
\Lambda_b\mathbf 1_{\{b\mid j\}}
+
\Lambda_a\mathbf 1_{\{a\mid j\}}
}{j^2}
+
O(j^{-3}),
\qquad j\to\infty.
\end{equation}
Thus the exact descent set determines where the drops occur, while
$\Lambda_a$ and $\Lambda_b$ determine their leading asymptotic
amplitudes.
\end{theorem}

\begin{proof}
By the second-order expansion \eqref{eq:Sigma-two-term},
\[
\Sigma_j(M_{k,\ell})
=
\frac{C_{k,\ell}}j
+
\frac{\Psi_{k,\ell}(j)}{j^2}
+
O(j^{-3}).
\]
Since $\Psi_{k,\ell}$ is periodic, it is bounded. Hence
\[
\begin{aligned}
\Sigma_{j+1}(M_{k,\ell})
=
\frac{C_{k,\ell}}{j+1}
+
\frac{\Psi_{k,\ell}(j+1)}{(j+1)^2}
+
O(j^{-3})
=
\frac{C_{k,\ell}}j
-
\frac{C_{k,\ell}}{j^2}
+
\frac{\Psi_{k,\ell}(j+1)}{j^2}
+
O(j^{-3}).
\end{aligned}
\]
Hence
\[
\begin{aligned}
\Sigma_j(M_{k,\ell})
-
\Sigma_{j+1}(M_{k,\ell})
=
\frac{
C_{k,\ell}
+
\Psi_{k,\ell}(j)
-
\Psi_{k,\ell}(j+1)
}{j^2}
+
O(j^{-3}).
\end{aligned}
\]

Using \eqref{eq:Psi-second-order} together with
\eqref{eq:ceiling-increment}, we have
\[
\begin{aligned}
\Psi_{k,\ell}(j+1)-\Psi_{k,\ell}(j)
=
\mathscr C_1(\alpha,\beta)
\left(
-\frac1b+\mathbf 1_{\{b\mid j\}}
\right)
+
\mathscr C_2(\alpha,\beta)
\left(
-\frac1a+\mathbf 1_{\{a\mid j\}}
\right).
\end{aligned}
\]
Using
\[
C_{k,\ell}
=
-\frac{\mathscr C_1(\alpha,\beta)}b
-
\frac{\mathscr C_2(\alpha,\beta)}a,
\]
we obtain
\[
\begin{aligned}
C_{k,\ell}
+
\Psi_{k,\ell}(j)
-
\Psi_{k,\ell}(j+1)
=
-\mathscr C_1(\alpha,\beta)\mathbf 1_{\{b\mid j\}}
-
\mathscr C_2(\alpha,\beta)\mathbf 1_{\{a\mid j\}}.
\end{aligned}
\]
The definitions of $\Lambda_b$ and $\Lambda_a$ now give the stated
formula.
\end{proof}

The following subsequence limits follow immediately from
\eqref{eq:drop-asymptotic}.

If
$a\mid j,
b\nmid j$,
then, along this infinite subsequence,
\[
j^2
\bigl(
\Sigma_j(M_{k,\ell})-\Sigma_{j+1}(M_{k,\ell})
\bigr)
\longrightarrow
\Lambda_a.
\]

If $a\nmid b$, then there are infinitely many indices satisfying
$b\mid j,
a\nmid j$,
and along this subsequence,
\[
j^2
\bigl(
\Sigma_j(M_{k,\ell})-\Sigma_{j+1}(M_{k,\ell})
\bigr)
\longrightarrow
\Lambda_b.
\]

Along the common multiples,
$\operatorname{lcm}(a,b)\mid j$,
one has
\[
j^2
\bigl(
\Sigma_j(M_{k,\ell})-\Sigma_{j+1}(M_{k,\ell})
\bigr)
\longrightarrow
\Lambda_a+\Lambda_b.
\]

If neither $a$ nor $b$ divides $j$, the preceding work
\cite[Theorem 3.6]{LiuLuZu2026Quasi} gives
$\Sigma_j(M_{k,\ell})
=
\Sigma_{j+1}(M_{k,\ell})$
exactly.  


To relate these local amplitudes to the global tail coefficient,
recall that
$L=\operatorname{lcm}(a,b)$
and define 
\[
\mathcal A_{k,\ell}(j)
=
\Lambda_b\,\mathbf 1_{\{b\mid j\}}
+
\Lambda_a\,\mathbf 1_{\{a\mid j\}}.
\]

\begin{corollary}
\label{cor:mean-descent}
The mean of the leading local descent amplitudes over one period is
the leading tail coefficient:
\[
\frac1L
\sum_{j=1}^{L}
\mathcal A_{k,\ell}(j)
=
C_{k,\ell}.
\]
\end{corollary}

\begin{proof}
Among $1,\ldots,L$, exactly $L/b$ integers are divisible by $b$ and
exactly $L/a$ are divisible by $a$.  Hence
\[
\begin{aligned}
\frac1L
\sum_{j=1}^{L}
\mathcal A_{k,\ell}(j)
=
\frac{\Lambda_b}{b}
+
\frac{\Lambda_a}{a}
=
C_{k,\ell},
\end{aligned}
\]
where the last equality follows from \eqref{eq:Euler-C}.
\end{proof}

Equivalently,
\[
\Psi_{k,\ell}(j+1)-\Psi_{k,\ell}(j)
=
C_{k,\ell}-\mathcal A_{k,\ell}(j).
\]
Thus \(\Psi_{k,\ell}\) is a periodic discrete primitive of the
centered function
$C_{k,\ell}-\mathcal A_{k,\ell}(j)$.


\section{Yang's generating function and the polylogarithmic boundary hierarchy}
\label{sec5}

Recall that Yang's generating function for $M_{k,\ell}$ is
\[
\mathcal P_{k,\ell}(t)
:=
p_{M_{k,\ell}}(t)
=
\sum_{j=0}^{\infty}\Sigma_j(M_{k,\ell})t^j.
\]
We now examine how the first- and second-order tail asymptotics
obtained above are reflected in the boundary behavior of
$\mathcal P_{k,\ell}$.

Throughout Sections~\ref{sec5}--\ref{sec6}, $\operatorname{Log}$ denotes the principal branch of the
complex logarithm. Since $1-t$ lies in the open right half-plane for
$t\in\mathbb D$, we have
\[
-\operatorname{Log}(1-t)
=
\operatorname{Log}\frac{1}{1-t}
=
\sum_{j=1}^{\infty}\frac{t^j}{j},
\qquad |t|<1.
\]
When the argument is positive real, we use $\log$ for the ordinary
real natural logarithm.

\subsection{The leading logarithmic singularity}

Recall from Theorem~\ref{thm:tail-main} that
\[
\Sigma_j(M_{k,\ell})
=
\frac{C_{k,\ell}}{j}
+
O(j^{-2}),
\qquad
C_{k,\ell}>0.
\]

Recall from \eqref{eq:renormalized-remainder}--\eqref{eq:renormalized-remainder-bound} that
\[
r_j^{(k,\ell)}
=
\Sigma_j(M_{k,\ell})-\frac{C_{k,\ell}}{j}
=
O(j^{-2}).
\]
Hence
$\sum_{j\ge1}|r_j^{(k,\ell)}|<\infty$.

\begin{theorem}
\label{thm:generating-leading}
The power series $\mathcal P_{k,\ell}$ has radius of convergence
exactly one.  Moreover,
\[
\mathcal G_{k,\ell}(t)
:=
\mathcal P_{k,\ell}(t)
-
C_{k,\ell}\operatorname{Log}\frac{1}{1-t},
\qquad |t|<1,
\]
extends continuously to the closed unit disk.  Define
\[
K_{k,\ell}
=
\Sigma_0(M_{k,\ell})
+
\sum_{j=1}^{\infty}
\left(
\Sigma_j(M_{k,\ell})
-
\frac{C_{k,\ell}}{j}
\right).
\]
Then the continuous extension satisfies
$\mathcal G_{k,\ell}(1)=K_{k,\ell}$.

Consequently, for real \(t\in(0,1)\), as \(t\to1^{-}\),
\[
\mathcal P_{k,\ell}(t)
=
C_{k,\ell}\log\frac1{1-t}
+
K_{k,\ell}
+
O\left(
(1-t)\log\frac1{1-t}
\right).
\]
\end{theorem}

\begin{proof}
Since
$\Sigma_j(M_{k,\ell})
\sim
\frac{C_{k,\ell}}{j}$,
we have
$\lim_{j\to\infty}
\Sigma_j(M_{k,\ell})^{1/j}
=
1$.
The Cauchy--Hadamard formula therefore gives radius of convergence
one.

For $|t|<1$,
\[
\begin{aligned}
\mathcal P_{k,\ell}(t)
=
\Sigma_0(M_{k,\ell})
+
C_{k,\ell}
\sum_{j=1}^{\infty}\frac{t^j}{j}
+
\sum_{j=1}^{\infty}r_j^{(k,\ell)}t^j
=
C_{k,\ell}\operatorname{Log}\frac1{1-t}
+
\Sigma_0(M_{k,\ell})
+
\sum_{j=1}^{\infty}r_j^{(k,\ell)}t^j.
\end{aligned}
\]
The last series converges uniformly on $|t|\leq1$, since
$\sum_{j=1}^{\infty}
|r_j^{(k,\ell)}|
<\infty$.
Hence $\mathcal G_{k,\ell}$ extends continuously to the closed unit
disk, with
$\mathcal G_{k,\ell}(1)
=
K_{k,\ell}$.

It remains to estimate the approach to this boundary value.  Put
$s=1-t,
0<s<1$.
Since
$0\leq1-t^j\leq\min\{1,js\}$
and $r_j^{(k,\ell)}=O(j^{-2})$, we obtain
\[
\begin{aligned}
\left|
\mathcal G_{k,\ell}(t)-K_{k,\ell}
\right|
\leq
\sum_{j=1}^{\infty}
|r_j^{(k,\ell)}|
(1-t^j)
=
O\left(
s\sum_{j\leq s^{-1}}\frac1j
+
\sum_{j>s^{-1}}\frac1{j^2}
\right)
=
O\left(
s\log\frac1s
\right).
\end{aligned}
\]
Substituting $s=1-t$ proves the assertion.
\end{proof}

\begin{remark}
If $S_{k,\ell}$ denotes the constant appearing in the partial-sum
asymptotic from Section~\ref{sec3}, then
$S_{k,\ell}
=
K_{k,\ell}
+
C_{k,\ell}\gamma$.
Thus the constant term in the generating-function expansion differs
from the constant term in the logarithmic partial-sum asymptotic by
the Euler-constant contribution.
\end{remark}

\subsection{The periodic second-order term and its mean mode}
\label{sec5.2}

We now use the second-order expansion
\eqref{eq:Sigma-two-term}.  Recall that
$L=\operatorname{lcm}(k,\ell)$
and set
$\omega=e^{2\pi i/L}$.

The second-order coefficient satisfies
$\Psi_{k,\ell}(j+L)
=
\Psi_{k,\ell}(j)$.
Define its discrete Fourier coefficients by
\[
\widehat{\Psi}_q
=
\frac1L
\sum_{m=1}^{L}
\Psi_{k,\ell}(m)\omega^{-qm},
\qquad
0\leq q\leq L-1.
\]
Then
\[
\Psi_{k,\ell}(j)
=
\sum_{q=0}^{L-1}
\widehat{\Psi}_q\omega^{qj}.
\]

We use the standard dilogarithm
\[
\operatorname{Li}_2(z)
=
\sum_{j=1}^{\infty}
\frac{z^j}{j^2},
\qquad
|z|<1;
\]
see, for example,
\cite[Sec.~25.12(i), Eq.~(25.12.1)]{NIST2010}.

Define
\[
e_j^{(k,\ell)}
=
\Sigma_j(M_{k,\ell})
-
\frac{C_{k,\ell}}{j}
-
\frac{\Psi_{k,\ell}(j)}{j^2}.
\]
By Theorem~\ref{thm:second-order-periodic},
$e_j^{(k,\ell)}
=
O(j^{-3})$.

\begin{lemma}
\label{lem:dilogarithmic-decomposition}
For $|t|<1$,
\begin{equation}\label{eq:Li2-decomposition}
\mathcal P_{k,\ell}(t)=
C_{k,\ell}\operatorname{Log}\frac1{1-t}
+
\Sigma_0(M_{k,\ell})+
\sum_{q=0}^{L-1}
\widehat{\Psi}_q
\operatorname{Li}_2(\omega^qt)
+
\mathcal E_{k,\ell}(t),
\qquad |t|<1,
\end{equation}
where
$\mathcal E_{k,\ell}(t)
=
\sum_{j=1}^{\infty}
e_j^{(k,\ell)}t^j$.

Moreover, $\mathcal E_{k,\ell}$ and
$\mathcal E_{k,\ell}'$ extend continuously to the closed unit disk.
\end{lemma}

\begin{proof}
By \eqref{eq:Sigma-two-term},
\[
\begin{aligned}
\mathcal P_{k,\ell}(t)
=
\Sigma_0(M_{k,\ell})
+
C_{k,\ell}
\sum_{j=1}^{\infty}\frac{t^j}{j}
+
\sum_{j=1}^{\infty}
\frac{\Psi_{k,\ell}(j)t^j}{j^2}
+
\sum_{j=1}^{\infty}
e_j^{(k,\ell)}t^j.
\end{aligned}
\]
Using the finite Fourier expansion of $\Psi_{k,\ell}$,
\[
\begin{aligned}
\sum_{j=1}^{\infty}
\frac{\Psi_{k,\ell}(j)t^j}{j^2}
=
\sum_{q=0}^{L-1}
\widehat{\Psi}_q
\sum_{j=1}^{\infty}
\frac{(\omega^qt)^j}{j^2}=
\sum_{q=0}^{L-1}
\widehat{\Psi}_q
\operatorname{Li}_2(\omega^qt).
\end{aligned}
\]
This proves the decomposition.

Since
$e_j^{(k,\ell)}
=
O(j^{-3})$,
both series
$\sum_{j=1}^{\infty}|e_j^{(k,\ell)}|$
and
$\sum_{j=1}^{\infty}j|e_j^{(k,\ell)}|$
converge.  Hence the series defining
$\mathcal E_{k,\ell}$ and its first derivative converge uniformly on
the closed unit disk.
\end{proof}


The zero Fourier mode of the periodic second-order coefficient admits
a particularly simple form.

\begin{lemma}
\label{lem:mean-Psi}
One has
\begin{equation}\label{eq:Psi-zero-mode}
\widehat{\Psi}_0
=
\frac1L
\sum_{j=1}^{L}
\Psi_{k,\ell}(j)
=
\frac{C_{k,\ell}}{2}.
\end{equation}
\end{lemma}

\begin{proof}
As in Section~\ref{sec4}, set
\[
a=\min\{k,\ell\},
\qquad
b=\max\{k,\ell\},
\qquad
\alpha=\frac1b,
\qquad
\beta=\frac1a.
\]
By \eqref{eq:Psi-second-order}, the second-order coefficient is
\[
\Psi_{k,\ell}(j)
=
\mathscr Q(\alpha,\beta)
+
\varepsilon_b(j)\mathscr C_1(\alpha,\beta)
+
\varepsilon_a(j)\mathscr C_2(\alpha,\beta).
\]
Since
\[
\frac1m
\sum_{j=1}^{m}
\varepsilon_m(j)
=
\frac{m-1}{2m},
\]
averaging over a common period gives
\[
\widehat{\Psi}_0
=
\mathscr Q(\alpha,\beta)
+
\frac{1-\alpha}{2}\mathscr C_1(\alpha,\beta)
+
\frac{1-\beta}{2}\mathscr C_2(\alpha,\beta).
\]

Recall from \eqref{eq:Q-def} that
\[
\mathscr Q(\alpha,\beta)
=
\frac{(\alpha+\beta)\log(\beta/\alpha)-2(\beta-\alpha)}
{(\beta-\alpha)^3}.
\]
We next relate this quantity to the first derivatives of
\(\mathscr C\).  A direct partial-fraction calculation gives
\[
\int_0^\infty
\frac{x\,dx}
{(x+\alpha)^2(x+\beta)^2}
=
\mathscr Q(\alpha,\beta).
\]
Indeed, with $d=\beta-\alpha$,
\[
\begin{aligned}
\frac{x}
{(x+\alpha)^2(x+\beta)^2}
=
-\frac{\alpha}{d^2(x+\alpha)^2}
-\frac{\beta}{d^2(x+\beta)^2}+
\frac{\alpha+\beta}{d^3}
\left(
\frac1{x+\alpha}
-
\frac1{x+\beta}
\right),
\end{aligned}
\]
and integration yields
\[
\frac{
(\alpha+\beta)\log(\beta/\alpha)-2d
}{d^3}
=
\mathscr Q(\alpha,\beta).
\]

On the other hand,
\[
f(x)
=
\frac{x^2}
{(x+\alpha)^2(x+\beta)^2}
\]
satisfies
\[
\begin{aligned}
f'(x)
=
\frac{2x}
{(x+\alpha)^2(x+\beta)^2}
-
\frac{2x^2}
{(x+\alpha)^3(x+\beta)^2}
-
\frac{2x^2}
{(x+\alpha)^2(x+\beta)^3}.
\end{aligned}
\]
Since $f(0)=0$ and $f(x)\to0$ as $x\to\infty$, integration over
$(0,\infty)$ gives
\[
\mathscr C_1(\alpha,\beta)
+
\mathscr C_2(\alpha,\beta)
=
-2\mathscr Q(\alpha,\beta).
\]

Recall from \eqref{eq:Euler-C} that
\[
\alpha\mathscr C_1(\alpha,\beta)
+
\beta\mathscr C_2(\alpha,\beta)
=
-\mathscr C(\alpha,\beta).
\]
Therefore
\[
\begin{aligned}
\widehat{\Psi}_0
=
\mathscr Q
+
\frac12(\mathscr C_1+\mathscr C_2)
-
\frac12(\alpha\mathscr C_1+\beta\mathscr C_2)=
\mathscr Q-\mathscr Q
+
\frac12\mathscr C(\alpha,\beta)
=
\frac{C_{k,\ell}}2.
\end{aligned}
\]
\end{proof}

\subsection{Refined boundary behavior at \texorpdfstring{\(t=1\)}{t=1}}

We can now sharpen the error term in
Theorem~\ref{thm:generating-leading}.

\begin{theorem}
\label{thm:generating-refined}
For real \(t\in(0,1)\), as \(t\to1^{-}\),
\begin{equation}\label{eq:refined-boundary-expansions}
\begin{aligned}
\mathcal P_{k,\ell}(t)
&=
C_{k,\ell}\log\frac{1}{1-t}
+
K_{k,\ell}
-
\frac{C_{k,\ell}}{2}
(1-t)\log\frac{1}{1-t}
+
O(1-t),\\
\mathcal P'_{k,\ell}(t)
&=
\frac{C_{k,\ell}}{1-t}
+
\frac{C_{k,\ell}}{2}
\log\frac{1}{1-t}
+
O(1).
\end{aligned}
\end{equation}
\end{theorem}

\begin{proof}
By \eqref{eq:Psi-zero-mode}, the zero Fourier mode is
\[
\widehat{\Psi}_0\operatorname{Li}_2(t)
=
\frac{C_{k,\ell}}2\operatorname{Li}_2(t).
\]

By the classical dilogarithm identity
\[
\operatorname{Li}_2(t)+\operatorname{Li}_2(1-t)
=
\frac{\pi^2}{6}
-\log t\,\log(1-t),
\qquad 0<t<1,
\]
see \cite[Sec.~25.12(i), Eq.~(25.12.6)]{NIST2010}, the standard expansion of the dilogarithm at $1$ is
\[
\operatorname{Li}_2(t)
=
\frac{\pi^2}{6}
-
(1-t)\log\frac1{1-t}
-
(1-t)
+
O\left(
(1-t)^2\log\frac1{1-t}
\right).
\]

For $q\neq0$, one has $1-\omega^q\neq0$, and hence
$\operatorname{Li}_2(\omega^qt)$ and its first derivative remain
finite as \(t\to1^{-}\).  Their variation from the value at $t=1$ is
therefore $O(1-t)$.  Also,
$\mathcal E_{k,\ell}\in C^1(\overline{\mathbb D})$, so
\[
\mathcal E_{k,\ell}(t)
=
\mathcal E_{k,\ell}(1)
+
O(1-t).
\]
Combining these facts with
Theorem~\ref{thm:generating-leading} gives the first expansion.

Finally,
\[
\frac{d}{dt}\operatorname{Li}_2(t)
=
-\frac{\log(1-t)}{t}
=
\frac1t\log\frac1{1-t}.
\]
Thus the zero Fourier mode contributes
\[
\frac{C_{k,\ell}}2
\log\frac1{1-t}
+
O(1)
\]
to the derivative.  The derivatives of all nonzero Fourier modes and
of $\mathcal E_{k,\ell}$ remain bounded at $t=1$, while
\[
\frac{d}{dt}
\left(
C_{k,\ell}\operatorname{Log}\frac1{1-t}
\right)
=
\frac{C_{k,\ell}}{1-t}.
\]
The derivative expansion follows.
\end{proof}

\begin{corollary}
\label{cor:generating-coefficient}
One has
\begin{equation}\label{eq:C-boundary-recovery}
C_{k,\ell}
=
\lim_{t\to1^-}
\frac{\mathcal P_{k,\ell}(t)}
{\log\frac{1}{1-t}}
=
\lim_{t\to1^-}
(1-t)\mathcal P'_{k,\ell}(t).
\end{equation}
Consequently, the asymptotic rigidity theorem of
Section~\ref{sec3} may equivalently be formulated
using the coefficient of the logarithmic boundary singularity of
\(\mathcal P_{k,\ell}\).
\end{corollary}

\begin{proof}
Both limits follow immediately from the two expansions in
\eqref{eq:refined-boundary-expansions}.  The final statement is then
Theorem~\ref{thm:asymptotic-rigidity} with $C_{k,\ell}$ read from the
boundary behavior of the generating function.
\end{proof}

Since $C_{k,\ell}>0$, the generating function has a genuine
logarithmic divergence at $t=1$. In particular, $P_{k,\ell}$ is not
a polynomial for any $k\neq\ell$, thereby giving a negative answer
to Yang's polynomiality question within this quasi-homogeneous family.

\subsection{Residue-class expansions}

For $m\in\mathbb N$, recall
\[
\varepsilon_m(j)
=
\left\lceil\frac jm\right\rceil-\frac jm.
\]
Since $L$ is divisible by both $a$ and $b$, the quantities
$\varepsilon_a(j)$ and $\varepsilon_b(j)$ are constant on every fixed
residue class modulo $L$.

More precisely, fix a residue class \(r\) modulo \(L\) and put
$c_r=\varepsilon_b(r),
d_r=\varepsilon_a(r)$,
where, for the residue class \(r=0\), we set
$c_0=d_0=0$.

Then, whenever
$j\equiv r\pmod L$,
\[
u_j
=
\left\lceil\frac jb\right\rceil
=
\alpha j+c_r
\]
and
\[
\eta_j
=
\left\lceil\frac ja\right\rceil
=
\beta j+d_r.
\]

We shall also use the standard complete asymptotic expansion of the
harmonic numbers.  For every fixed $M\geq1$,
\[
H_{n-1}
=
\log n+\gamma-\frac1{2n}
-
\sum_{s=1}^{M}
\frac{B_{2s}}{2s\,n^{2s}}
+
O(n^{-2M-2}),
\qquad
n\to\infty,
\]
where $B_{2s}$ denotes the Bernoulli numbers.  This follows from
$H_{n-1}=\psi(n)+\gamma$
and the standard asymptotic expansion of the digamma function; see
\cite[Sec.~5.4(ii), Eq.~(5.4.14), and Sec.~5.11(i),
Eq.~(5.11.2)]{NIST2010}.

\begin{theorem}[Periodic Poincar\'e asymptotic hierarchy]
\label{thm:periodic-Poincare-expansion}
There exists a unique sequence of real-valued \(L\)-periodic functions
$V_m^{(k,\ell)}:\mathbb N\longrightarrow\mathbb R,
 m\geq1$,
with the following property: for every fixed integer \(N\geq1\),
\begin{equation}\label{eq:periodic-Poincare}
\Sigma_j(M_{k,\ell})
=
\sum_{m=1}^{N}
\frac{V_m^{(k,\ell)}(j)}{j^m}
+
O(j^{-N-1}),
\qquad j\to\infty.
\end{equation}
The first two coefficient functions are
\[
V_1^{(k,\ell)}(j)\equiv C_{k,\ell},
\qquad
V_2^{(k,\ell)}(j)=\Psi_{k,\ell}(j),
\]
where \(\Psi_{k,\ell}\) is the periodic second-order coefficient
introduced in Section~\ref{sec4}.
\end{theorem}

\begin{proof}
Fix a residue class $r\in\{0,1,\ldots,L\}$ and let $j\to\infty$ along the
residue class $j\equiv r\pmod L$. Then
$u_j=\alpha j+c_r,
\eta_j=\beta j+d_r$,
where $c_r$ and $d_r$ depend only on the fixed residue class $r$.

Since
$\beta-\alpha>0$,
we have
$\eta_j-u_j
=
(\beta-\alpha)j+(d_r-c_r)
\longrightarrow\infty$.
Hence the separated-parameter formula \eqref{eq:F-separated} applies for all sufficiently large $j$ in
this residue class.

Let \(N\geq1\) be fixed. Using the complete asymptotic expansion of
\(H_{n-1}\) displayed above, truncate it at a sufficiently high order. Choose $M$ sufficiently large so that
$2M+3\ge N+1$.
The remainder $O(n^{-2M-2})$ in the harmonic-number expansion,
after multiplication by the quadratic factor in the numerator of
the separated-parameter formula and division by its cubic
denominator, contributes $O(j^{-2M-3})=O(j^{-N-1})$. 

Substitution of
\[
u_j=\alpha j+c_r,
\qquad
\eta_j=\beta j+d_r
\]
gives an expansion of
$H_{\eta_j-1}-H_{u_j-1}$
in integral powers of $j^{-1}$ to any prescribed finite order.

Indeed,
\[
\log(\alpha j+c_r)
=
\log j+\log\alpha
+
\log\left(
1+\frac{c_r}{\alpha j}
\right)
\]
and
\[
\log(\beta j+d_r)
=
\log j+\log\beta
+
\log\left(
1+\frac{d_r}{\beta j}
\right).
\]
The two $\log j$ terms cancel in their difference, while each of the
remaining logarithms has an ordinary Taylor expansion in powers of
$j^{-1}$.  The inverse powers of $u_j$ and $\eta_j$ occurring in the
Euler--Maclaurin expansion likewise admit ordinary expansions in
$j^{-1}$.

The other factors in the separated formula are polynomial expressions
in $u_j$ and $\eta_j$.  Its denominator is
$(\eta_j-u_j)
\bigl(
(\eta_j-u_j)^2-1
\bigr)$,
which is a cubic polynomial in $j$ with leading coefficient
$(\beta-\alpha)^3>0$.
Its reciprocal therefore has a Poincar\'e expansion in powers of
$j^{-1}$ to arbitrary finite order. All the finite Taylor expansions involved above may be truncated
at sufficiently high order so that their combined remainder is
$O(j^{-N-1})$.

Combining these expansions in the separated formula gives, along the
fixed residue class,
\[
\Sigma_j(M_{k,\ell})
=
\sum_{m=1}^{N}
\frac{A_{m,r}}{j^m}
+
O(j^{-N-1}),
\]
where the constants $A_{m,r}$ depend only on the residue class $r$.

The coefficients obtained in this way are consistent as the truncation
order \(N\) varies. Indeed, increasing the truncation order cannot alter
any coefficient already determined at a lower order, by the uniqueness
of asymptotic expansions in powers of \(j^{-1}\) along the fixed residue
class.

There are only finitely many residue classes modulo \(L\). Define
\[
V_m^{(k,\ell)}(j)=A_{m,r}
\]
whenever
$j\equiv r\pmod L$.
Then every \(V_m^{(k,\ell)}\) is \(L\)-periodic.  Since the number of
residue classes is finite, the remainder constants may be chosen
uniformly by taking their maximum.  This proves the asserted
expansion.

The identification of the first coefficient follows from
Theorem~\ref{thm:tail-main}, which gives
\[
\Sigma_j(M_{k,\ell})
=
\frac{C_{k,\ell}}j
+
O(j^{-2}).
\]
Thus
$V_1^{(k,\ell)}
=
C_{k,\ell}$.
Similarly, \eqref{eq:Sigma-two-term} gives
\[
\Sigma_j(M_{k,\ell})
=
\frac{C_{k,\ell}}j
+
\frac{\Psi_{k,\ell}(j)}{j^2}
+
O(j^{-3}),
\]
and hence
$V_2^{(k,\ell)}(j)=\Psi_{k,\ell}(j)$.

It remains to prove uniqueness. Fix a residue class
$r\in\{1,\ldots,L\}$. Suppose that two collections of coefficients,
$\{A_{m,r}\}_{m\geq1}$ and
$\{\widetilde A_{m,r}\}_{m\geq1}$, give asymptotic expansions
of the stated form along $j\equiv r\pmod L$.
If the two collections are not identical, let $m_0$ be the smallest
index such that
$A_{m_0,r}\neq \widetilde A_{m_0,r}$.
Subtracting the two expansions through order $m_0$ gives
\[
\frac{A_{m_0,r}-\widetilde A_{m_0,r}}{j^{m_0}}
=
O(j^{-m_0-1})
\]
as $j\to\infty$ along $j\equiv r\pmod L$.
Multiplying by $j^{m_0}$ and letting $j\to\infty$ along this
residue class yields
$A_{m_0,r}-\widetilde A_{m_0,r}=0$,
contradicting the choice of $m_0$. Hence the coefficients are unique
on each residue class. Therefore all the periodic coefficient
functions $V_m^{(k,\ell)}$ are uniquely determined.
\end{proof}

\subsection{The associated polylogarithmic hierarchy}

For every $m\geq1$, the periodic coefficient
$V_m^{(k,\ell)}$ has a finite Fourier expansion.  Retain
$\omega=e^{2\pi i/L}$
and define
\begin{equation}\label{eq:higher-Fourier-expansion}
\widehat V_{m,q}
=
\frac1L
\sum_{r=1}^{L}
V_m^{(k,\ell)}(r)\omega^{-qr},
\qquad
0\leq q\leq L-1,
\end{equation}
so that
\[
V_m^{(k,\ell)}(j)
=
\sum_{q=0}^{L-1}
\widehat V_{m,q}\omega^{qj}.
\]
In particular, since
$V_2^{(k,\ell)}=\Psi_{k,\ell}$,
the notation introduced in Section~\ref{sec5.2} gives
\[
\widehat V_{2,q}=\widehat{\Psi}_q,
\qquad
0\leq q\leq L-1.
\]

For $m\geq1$, let
\[
\operatorname{Li}_m(z)
=
\sum_{j=1}^{\infty}\frac{z^j}{j^m},
\qquad
|z|<1.
\]
For $m=1$,
\[
\operatorname{Li}_1(z)
=
-\operatorname{Log}(1-z)
=
\operatorname{Log}\frac{1}{1-z}.
\]

\begin{theorem}
\label{thm:polylogarithmic-hierarchy}
For every fixed integer $N\geq1$ and every $|t|<1$,
\begin{equation}\label{eq:polylog-hierarchy}
\mathcal P_{k,\ell}(t)
=
\Sigma_0(M_{k,\ell})
+
\sum_{m=1}^{N}
\sum_{q=0}^{L-1}
\widehat V_{m,q}
\operatorname{Li}_m(\omega^q t)
+
\mathcal E_N^{(k,\ell)}(t),
\end{equation}
where
$\mathcal E_N^{(k,\ell)}(t)
=
\sum_{j=1}^{\infty}
e_{N,j}^{(k,\ell)}t^j$
and
$e_{N,j}^{(k,\ell)}
=
O(j^{-N-1})$.

Moreover,
\[
\mathcal E_N^{(k,\ell)}
\in
C^{N-1}(\overline{\mathbb D}).
\]
\end{theorem}

\begin{proof}
By the periodic Poincar\'e expansion
\eqref{eq:periodic-Poincare}, we may write
\[
\Sigma_j(M_{k,\ell})
=
\sum_{m=1}^{N}
\frac{V_m^{(k,\ell)}(j)}{j^m}
+
e_{N,j}^{(k,\ell)},
\]
where
$e_{N,j}^{(k,\ell)}=O(j^{-N-1})$.

Hence, for $|t|<1$,
\[
\begin{aligned}
\mathcal P_{k,\ell}(t)
=
\Sigma_0(M_{k,\ell})
+
\sum_{m=1}^{N}
\sum_{j=1}^{\infty}
\frac{V_m^{(k,\ell)}(j)t^j}{j^m}
+
\sum_{j=1}^{\infty}
e_{N,j}^{(k,\ell)}t^j.
\end{aligned}
\]
Using the finite Fourier expansion
\eqref{eq:higher-Fourier-expansion},
\[
\begin{aligned}
\sum_{j=1}^{\infty}
\frac{V_m^{(k,\ell)}(j)t^j}{j^m}
=
\sum_{q=0}^{L-1}
\widehat{V}_{m,q}
\sum_{j=1}^{\infty}
\frac{(\omega^qt)^j}{j^m}
=
\sum_{q=0}^{L-1}
\widehat{V}_{m,q}
\operatorname{Li}_m(\omega^qt).
\end{aligned}
\]
This proves the stated decomposition.

It remains to establish the boundary regularity of the remainder.
Let
$0\leq r\leq N-1$.
After differentiating the remainder $r$ times, its coefficients are
bounded by a constant multiple of
$j^r
\left|
e_{N,j}^{(k,\ell)}
\right|$.

Since
$j^r
e_{N,j}^{(k,\ell)}
=
O(j^{r-N-1})$
and
$r-N-1\leq-2$,
the resulting coefficient series is absolutely summable.  The
Weierstrass $M$-test therefore yields uniform convergence on
$\overline{\mathbb D}$ for the remainder and each of its derivatives
of orders at most $N-1$.  Hence
$\mathcal E_N^{(k,\ell)}
\in
C^{N-1}(\overline{\mathbb D})$.
\end{proof}

The conjugate symmetry of the Fourier coefficients ensures that
\eqref{eq:polylog-hierarchy} is real-valued for \(0<t<1\).

Thus each periodic coefficient at order $j^{-m}$ gives rise to a
finite family of root-of-unity polylogarithmic modes, with higher-order
truncation yielding correspondingly stronger boundary regularity.

For every nontrivial \(L\)-th root of unity
\(\zeta=\omega^{-q}\neq1\), the preceding decomposition also shows
that \(P_{k,\ell}(r\zeta)\) has a finite radial limit as
\(r\to1^-\).  Nevertheless, the individual Fourier modes remain
detectable through logarithmic growth of suitable derivatives.
This leads naturally to the boundary reconstruction problem studied
in the next section.


\section{Boundary rigidity and reconstruction of the asymptotic hierarchy}
\label{sec6} 

We now prove the converse: boundary data of Yang's generating function
recover the periodic coefficients in~\eqref{eq:periodic-Poincare}, beginning with parameter
recovery from the second-order boundary data. Retain $L=\operatorname{lcm}(k,\ell)$ and
$\omega=e^{2\pi i/L}$.

\subsection{Fourier structure of the ceiling errors}

For $m\mid L$, regard the ceiling error
\[
\varepsilon_m(j)
=
\left\lceil\frac jm\right\rceil-\frac jm
\]
as an $L$-periodic function.  Define its Fourier coefficients by
\[
\widehat{\varepsilon}_{m,q}
=
\frac1L
\sum_{j=1}^{L}
\varepsilon_m(j)\omega^{-qj},
\qquad
0\leq q\leq L-1.
\]

\begin{lemma}
\label{lem:epsilon-Fourier}
Let \(m\mid L\), and set
\(\xi_s=e^{-2\pi i s/m}\) for \(1\leq s\leq m-1\).
Then the Fourier coefficients of \(\varepsilon_m\) satisfy
\begin{equation}\label{eq:epsilon-Fourier-explicit}
\widehat{\varepsilon}_{m,q}
=
\begin{cases}
\dfrac{m-1}{2m},
& q=0,
\\[2mm]
\dfrac{\xi_s}{m(1-\xi_s)},
& q=s\dfrac{L}{m},
\quad 1\leq s\leq m-1,
\\[2mm]
0,
& \text{otherwise}.
\end{cases}
\end{equation}
In particular, every nonzero Fourier mode allowed by the period
\(m\) is nonvanishing.
\end{lemma}

\begin{proof}
Write
$L=Mm$.
Since \(\varepsilon_m\) is \(m\)-periodic, every
\(j\in\{1,\ldots,L\}\) can be written uniquely as
\[
j=r+hm,
\qquad
1\leq r\leq m,
\quad
0\leq h\leq M-1.
\]
Hence
\[
\begin{aligned}
\widehat{\varepsilon}_{m,q}
=
\frac1L
\sum_{r=1}^{m}
\sum_{h=0}^{M-1}
\varepsilon_m(r+hm)\,
\omega^{-q(r+hm)}
=
\frac1L
\sum_{r=1}^{m}
\varepsilon_m(r)\omega^{-qr}
\sum_{h=0}^{M-1}
\omega^{-qhm}.
\end{aligned}
\]

Since
$(\omega^{-qm})^M
=
\omega^{-qL}
=
1$,
we have
\[
\sum_{h=0}^{M-1}\omega^{-qhm}
=
\begin{cases}
M, & \omega^{-qm}=1,\\[1mm]
0, & \omega^{-qm}\neq1.
\end{cases}
\]
Thus \(\widehat{\varepsilon}_{m,q}=0\) unless $\omega^{-qm}=1$, which is equivalent to
\[
q=s\frac{L}{m},
\qquad
s=0,\ldots,m-1.
\]

For $q=0$,
\[
\begin{aligned}
\widehat{\varepsilon}_{m,0}
=
\frac1m
\sum_{r=1}^{m}
\varepsilon_m(r)
=
\frac1{m^2}
\sum_{r=1}^{m-1}(m-r)
=
\frac{m-1}{2m}.
\end{aligned}
\]

Now let
\[
q=s\frac Lm,
\qquad
1\leq s\leq m-1,
\]
and put
$\xi_s=e^{-2\pi i s/m}$.
Then
\[
\begin{aligned}
\widehat{\varepsilon}_{m,q}
=
\frac1m
\sum_{r=1}^{m}
\varepsilon_m(r)\xi_s^r
=
\frac1{m^2}
\sum_{r=1}^{m-1}
(m-r)\xi_s^r.
\end{aligned}
\]

Since $\xi_s^m=1$ and $\xi_s\neq1$,
$\sum_{r=1}^{m-1}\xi_s^r=-1$
and
\[
\sum_{r=1}^{m-1}r\xi_s^r
=
-\frac{m}{1-\xi_s}.
\]
Therefore
\[
\begin{aligned}
\sum_{r=1}^{m-1}
(m-r)\xi_s^r
=
-m+\frac{m}{1-\xi_s}
=
\frac{m\xi_s}{1-\xi_s}.
\end{aligned}
\]
Substitution gives
\[
\widehat{\varepsilon}_{m,q}
=
\frac{\xi_s}{m(1-\xi_s)},
\]
which is nonzero.
\end{proof}


\subsection{Second-order boundary support and parameter recovery}

We now apply Lemma~\ref{lem:epsilon-Fourier} to the second-order
coefficient. Recall from Sections~\ref{sec4} and~\ref{sec5} that
\[
V_2^{(k,\ell)}(j)
=
\Psi_{k,\ell}(j)
=
\mathscr Q(\alpha,\beta)
+
\varepsilon_b(j)\mathscr C_1(\alpha,\beta)
+
\varepsilon_a(j)\mathscr C_2(\alpha,\beta),
\]
where $\alpha=\frac1b,
\beta=\frac1a$,
and
$\mathscr C_1(\alpha,\beta)<0,
\mathscr C_2(\alpha,\beta)<0$.

By \eqref{eq:higher-Fourier-expansion},
\[\widehat V_{2,q}
=
\frac1L
\sum_{j=1}^{L}
V_2^{(k,\ell)}(j)\omega^{-qj}.\]
Thus
$\widehat V_{2,q}=\widehat{\Psi}_q$.

For $q\neq0$,
\[
\widehat{V}_{2,q}
=
\mathscr C_1(\alpha,\beta)
\widehat{\varepsilon}_{b,q}
+
\mathscr C_2(\alpha,\beta)
\widehat{\varepsilon}_{a,q}.
\]

We remove the leading logarithmic level from Yang's generating
function by setting
\[
\mathcal R_2(t)
=
\mathcal P_{k,\ell}(t)
-
\Sigma_0(M_{k,\ell})
-
C_{k,\ell}\operatorname{Li}_1(t),
\]
where
\[
\operatorname{Li}_1(t)
=
\operatorname{Log}\frac1{1-t},
\qquad
|t|<1.
\]

Taking \(N=2\) in
\eqref{eq:polylog-hierarchy}, we obtain
\[
\mathcal R_2(t)
=
\sum_{q=0}^{L-1}
\widehat{V}_{2,q}
\operatorname{Li}_2(\omega^qt)
+
\mathcal E_2^{(k,\ell)}(t),
\]
where
$\mathcal E_2^{(k,\ell)}
\in
C^1(\overline{\mathbb D})$.

\begin{proposition}
\label{prop:second-boundary-spectrum}
For every \(\zeta\in\partial\mathbb D\), along the radial approach
\[
t=r\zeta,\qquad 0<r<1,\qquad r\to1^-,
\]
the normalized boundary limit
\[
\Gamma_2(\zeta)
:=
\lim_{r\to1^-}
\frac{\mathcal R_2'(r\zeta)}
{\log\frac1{1-r}}
\]
exists. Moreover,
\begin{equation}\label{eq:B2-Fourier-identification}
\Gamma_2(\zeta)
=
\begin{cases}
\omega^q\widehat{V}_{2,q},
& \zeta=\omega^{-q},
\quad 0\le q\le L-1,\\[6pt]
0,
& \zeta^L\ne1.
\end{cases}
\end{equation}
\end{proposition}

\begin{proof}
For \(|z|<1\),
\[
\frac{d}{dz}\operatorname{Li}_2(z)
=
\frac{\operatorname{Li}_1(z)}{z}.
\]
Hence
\[
\frac{d}{dt}\operatorname{Li}_2(\omega^p t)
=
\frac{\operatorname{Li}_1(\omega^p t)}{t}.
\]

First let
$\zeta=\omega^{-q},
 0\leq q\leq L-1$,
and approach \(\zeta\) radially by
$t=r\omega^{-q}, r\to1^-$.

For the Fourier mode \(p=q\),
$\omega^q t=r$,
and therefore
\[
\begin{aligned}
\frac{d}{dt}\operatorname{Li}_2(\omega^q t)
\bigg|_{t=r\omega^{-q}}
=
\frac{\operatorname{Li}_1(r)}
{r\omega^{-q}}
=
\frac{\omega^q}{r}
\log\frac1{1-r}.
\end{aligned}
\]

If \(p\neq q\), then
\[
\omega^p t
=
r\omega^{p-q}
\longrightarrow
\omega^{p-q}\neq1,
\]
so
\[
\frac{d}{dt}\operatorname{Li}_2(\omega^p t)
\bigg|_{t=r\omega^{-q}}
=
O(1)
\qquad (r\to1^-).
\]

Since
$\mathcal E_2^{(k,\ell)}
\in C^1(\overline{\mathbb D})$,
its derivative is bounded on \(\overline{\mathbb D}\).  Hence
\[
\mathcal R_2'(r\omega^{-q})
=
\frac{\omega^q}{r}\widehat V_{2,q}
\log\frac1{1-r}
+
O(1).
\]
Dividing by \(\log\frac1{1-r}\) and letting \(r\to1^-\), we obtain
$\Gamma_2(\omega^{-q})
=
\omega^q\widehat V_{2,q}$.

Now suppose that \(\zeta^L\neq1\), and approach \(\zeta\) radially
along
$t=r\zeta,
 r\to1^-$.
Then
$\omega^p\zeta\neq1$
for every $0\leq p\leq L-1$.

Thus every dilogarithmic derivative in the expansion of
\(\mathcal R_2'(r\zeta)\) remains bounded as \(r\to1^-\).
The derivative of the remainder is also bounded. Therefore
$\mathcal R_2'(r\zeta)=O(1)$,
and consequently
$\Gamma_2(\zeta)=0$.
This proves \eqref{eq:B2-Fourier-identification}.
\end{proof}

Define the second-order boundary support by
\[
\mathfrak B_2(k,\ell)
=
\left\{
\zeta\in\partial\mathbb D:
\Gamma_2(\zeta)\neq0
\right\}.
\]

\begin{theorem}
\label{thm:second-boundary-spectrum}
Every point of $\mathfrak B_2(k,\ell)$ is a root of unity whose order
divides either $a$ or $b$.  Moreover,
$e^{-2\pi i/b}
\in
\mathfrak B_2(k,\ell)$.
If $a\nmid b$, then also
$e^{-2\pi i/a}
\in
\mathfrak B_2(k,\ell)$. Consequently,
\begin{equation}\label{eq:b-L-boundary-recovery}
\begin{aligned}
b
&=
\max_{\zeta\in\mathcal B_2(k,\ell)}
\operatorname{ord}(\zeta),\\
L
&=
\operatorname{lcm}
\left\{
\operatorname{ord}(\zeta):
\zeta\in\mathcal B_2(k,\ell)
\right\}.
\end{aligned}
\end{equation}
\end{theorem}

\begin{proof}
By \eqref{eq:B2-Fourier-identification}, the points of
\(\mathfrak B_2(k,\ell)\) are precisely those \(L\)-th roots of unity
\(\omega^{-q}\), \(0\le q\le L-1\), for which
\(\widehat V_{2,q}\neq0\).

We first determine the possible orders of these points.
For \(q=0\), Lemma~\ref{lem:mean-Psi} gives
$\widehat V_{2,0}=\frac{C_{k,\ell}}{2}>0$.
Hence \(1\in\mathfrak B_2(k,\ell)\), and
\(\operatorname{ord}(1)=1\), which divides both \(a\) and \(b\).

Now let \(q\neq0\). Then
\[
\widehat V_{2,q}
=
\mathscr C_1(\alpha,\beta)\widehat\varepsilon_{b,q}
+
\mathscr C_2(\alpha,\beta)\widehat\varepsilon_{a,q}.
\]

If \(\widehat V_{2,q}\neq0\), then at least one of
\(\widehat\varepsilon_{b,q}\) and
\(\widehat\varepsilon_{a,q}\) is nonzero.
By \eqref{eq:epsilon-Fourier-explicit}, this implies that either
\[
q=s\frac{L}{b}
\qquad (1\le s\le b-1),
\]
or
\[
q=t\frac{L}{a}
\qquad (1\le t\le a-1).
\]
In the first case,
$\omega^{-q}=e^{-2\pi i s/b}$,
whose order divides \(b\); in the second case,
$\omega^{-q}=e^{-2\pi i t/a}$,
whose order divides \(a\).
Thus every point of \(\mathfrak B_2(k,\ell)\) has order dividing
either \(a\) or \(b\).

We next show that an element of order exactly \(b\) always occurs.
Set
$q_b=\frac{L}{b}$.
The coefficient
\(\widehat{\varepsilon}_{b,q_b}\)
is nonzero by
\eqref{eq:epsilon-Fourier-explicit}.  On the other hand,
$q_b$ is not a multiple of $L/a$, since such a relation would imply
that $a$ is a positive multiple of $b$, contradicting $a<b$.
Therefore
$\widehat{\varepsilon}_{a,q_b}=0$,
and therefore
$\widehat V_{2,q_b}
=
\mathscr C_1(\alpha,\beta)
\widehat\varepsilon_{b,q_b}
\neq0$.
Consequently,
\[
\omega^{-q_b}
=
e^{-2\pi i/b}
\in\mathfrak B_2(k,\ell),
\]
and this point has order exactly \(b\). Since every order occurring
in \(\mathfrak B_2(k,\ell)\) divides either \(a\) or \(b\), and
\(a<b\), it follows that
\[
b
=
\max_{\zeta\in\mathfrak B_2(k,\ell)}
\operatorname{ord}(\zeta).
\]

It remains to recover \(L\).
If \(a\mid b\), then
$L=\operatorname{lcm}(a,b)=b$.
Every order occurring in \(\mathfrak B_2(k,\ell)\) divides \(b\),
while an element of order \(b\) occurs. Hence
\[
L
=
\operatorname{lcm}
\{
\operatorname{ord}(\zeta):
\zeta\in\mathfrak B_2(k,\ell)
\}.
\]

Suppose now that \(a\nmid b\). Then necessarily \(a>1\). Set
$q_a=\frac{L}{a}$.
By \eqref{eq:epsilon-Fourier-explicit},
$\widehat\varepsilon_{a,q_a}\neq0$.
If \(q_a\) were a multiple of \(L/b\), then
$\frac{L}{a}
=
s\frac{L}{b}$
for some \(s\ge1\), which would imply \(b=sa\), contrary to
\(a\nmid b\). Thus
$\widehat\varepsilon_{b,q_a}=0$,
and hence
\[
\widehat V_{2,q_a}
=
\mathscr C_2(\alpha,\beta)
\widehat\varepsilon_{a,q_a}
\neq0.
\]
Therefore
\[
\omega^{-q_a}
=
e^{-2\pi i/a}
\in\mathfrak B_2(k,\ell),
\]
and this point has order exactly \(a\).

Hence, when \(a\nmid b\), the boundary support contains points of
orders \(a\) and \(b\), while every order occurring in it divides
either \(a\) or \(b\). Therefore
\[
L
=
\operatorname{lcm}(a,b)
=
\operatorname{lcm}
\{
\operatorname{ord}(\zeta):
\zeta\in\mathfrak B_2(k,\ell)
\}.
\]
This proves \eqref{eq:b-L-boundary-recovery}.
\end{proof}

\begin{corollary}
\label{cor:exact-second-period}
The least period of the second-order coefficient
$V_2^{(k,\ell)}=\Psi_{k,\ell}$
is exactly
$\operatorname{lcm}(k,\ell)$.
\end{corollary}

\begin{proof}
By the finite Fourier expansion \eqref{eq:higher-Fourier-expansion} and its
uniqueness, a positive integer \(d\) is a period of
\(V_2^{(k,\ell)}\) if and only if
$(\omega^q)^d=1$
for every \(q\) with \(\widehat V_{2,q}\neq0\). Hence its least period is
\[
\operatorname{lcm}
\left\{
\operatorname{ord}(\omega^{-q}):
\widehat V_{2,q}\neq0
\right\}.
\]
By \eqref{eq:B2-Fourier-identification}, these roots are precisely the
points of \(\mathfrak B_2(k,\ell)\), and the second identity in
\eqref{eq:b-L-boundary-recovery} shows that the above least common
multiple is
$L=\operatorname{lcm}(k,\ell)$.
\end{proof}


We now combine Theorem~\ref{thm:second-boundary-spectrum}
with \eqref{eq:C-boundary-recovery} to recover the remaining exponent.

\begin{theorem}[Boundary rigidity]
\label{thm:boundary-rigidity}
The unordered pair \(\{k,\ell\}\) is uniquely determined by the
boundary limit
\[
\lim_{t\to1^-}
\frac{\mathcal P_{k,\ell}(t)}
{\log\frac1{1-t}}
\]
together with the second-order boundary support
$\mathfrak B_2(k,\ell)$.
\end{theorem}

\begin{proof}
By \eqref{eq:C-boundary-recovery}, 
\[
C_{k,\ell}
=
\lim_{t\to1^-}
\frac{P_{k,\ell}(t)}
{\log\frac1{1-t}}.
\]

The first identity in \eqref{eq:b-L-boundary-recovery} determines
$b=\max\{k,\ell\}$
from the second-order boundary support. With
$a=\min\{k,\ell\}$,
we obtain
$C_{k,\ell}
=
\mathscr C\left(\frac1b,\frac1a\right)$.

For fixed \(b\), the function
\[
x\longmapsto
\mathscr C\left(\frac1b,\frac1x\right),
\qquad 0<x<b,
\]
is strictly increasing. Indeed, \(x\mapsto1/x\) is strictly
decreasing, whereas
$y\longmapsto \mathscr C\left(\frac1b,y\right)$
is strictly decreasing by Lemma~\ref{lem:C-monotone}.
Hence the equation
\[
\mathscr C\left(\frac1b,\frac1x\right)=C_{k,\ell}
\]
has at most one solution in \((0,b)\). Since the actual smaller
exponent \(a\) is such a solution, it is uniquely determined.
Thus the unordered pair \(\{k,\ell\}\) is uniquely determined.
\end{proof}

\begin{remark}
Unlike the asymptotic rigidity theorem in
Section~\ref{sec3}, Theorem~\ref{thm:boundary-rigidity}
uses only boundary data of Yang's generating function and requires no
descent information from the numerical invariant sequence.
\end{remark}

\subsection{Higher-order boundary reconstruction}

With $L$ already determined, we recover the higher-order periodic
coefficients successively using the following single-mode
polylogarithmic estimate.

\begin{lemma}
\label{lem:polylog-boundary-reconstruction}
Let $m\geq1$.  As $r\to1^-$,
\[
\left.
\frac{d^{m-1}}{dz^{m-1}}
\operatorname{Li}_m(z)
\right|_{z=r}
=
\log\frac1{1-r}
+
O(1).
\]
If $|\xi|=1$ and $\xi\neq1$, then
\[
\left.
\frac{d^{m-1}}{dz^{m-1}}
\operatorname{Li}_m(z)
\right|_{z=r\xi}
=
O(1).
\]
\end{lemma}

\begin{proof}
For \(m=1\), both assertions follow immediately from
$\operatorname{Li}_1(z)
=
\operatorname{Log}\frac1{1-z}$.
Assume \(m\ge2\). For \(|z|<1\),
\[
\operatorname{Li}_m(z)
=
\sum_{n=1}^{\infty}\frac{z^n}{n^m},
\]
and hence
\[
z^{m-1}
\frac{d^{m-1}}{dz^{m-1}}
\operatorname{Li}_m(z)
=
\sum_{n\ge m-1}
\frac{n(n-1)\cdots(n-m+2)}{n^m}z^n.
\]
For fixed \(m\),
\[
\frac{n(n-1)\cdots(n-m+2)}{n^m}
=
\frac1n
\prod_{j=1}^{m-2}
\left(1-\frac{j}{n}\right)
=
\frac1n+O(n^{-2}).
\]
Thus, after absorbing finitely many initial terms,
\[
z^{m-1}
\frac{d^{m-1}}{dz^{m-1}}
\operatorname{Li}_m(z)
=
\operatorname{Li}_1(z)+\mathcal H_m(z),
\]
where the Taylor coefficients of \(\mathcal H_m\) are
\(O(n^{-2})\). Hence
\(\mathcal H_m\in C(\overline{\mathbb D})\).

For \(z=r\),
\[
r^{m-1}
\frac{d^{m-1}}{dz^{m-1}}
\operatorname{Li}_m(z)\bigg|_{z=r}
=
\log\frac1{1-r}+O(1).
\]
Since
$r^{-(m-1)}=1+O(1-r)$
and
$(1-r)\log\frac1{1-r}\longrightarrow0$,
we obtain
\[
\frac{d^{m-1}}{dz^{m-1}}
\operatorname{Li}_m(z)\bigg|_{z=r}
=
\log\frac1{1-r}+O(1).
\]

If \(|\xi|=1\) and \(\xi\neq1\), then
$1-r\xi\longrightarrow1-\xi\neq0$,
so \(\operatorname{Li}_1(r\xi)\) remains bounded as \(r\to1^-\).
Since \(\mathcal H_m\) is also bounded on
\(\overline{\mathbb D}\),
\[
(r\xi)^{m-1}
\frac{d^{m-1}}{dz^{m-1}}
\operatorname{Li}_m(z)\bigg|_{z=r\xi}
=
O(1).
\]
Since \((r\xi)^{m-1}\) is bounded away from zero for \(r\)
sufficiently close to \(1\), the second assertion follows.
\end{proof}

For every $m\geq1$, define the successively renormalized generating
function
\[
\begin{aligned}
\mathcal R_m(t)
&=
\mathcal P_{k,\ell}(t)
-
\Sigma_0(M_{k,\ell})-
\sum_{s=1}^{m-1}
\sum_{p=0}^{L-1}
\widehat{V}_{s,p}
\operatorname{Li}_s(\omega^pt).
\end{aligned}
\]
For $m=1$, the double sum is empty.

By \eqref{eq:polylog-hierarchy}, after subtracting the first
\(m-1\) asymptotic levels we obtain
\[
\mathcal R_m(t)
=
\sum_{p=0}^{L-1}
\widehat{V}_{m,p}
\operatorname{Li}_m(\omega^pt)
+
\mathcal E_m^{(k,\ell)}(t),
\]
where
$\mathcal E_m^{(k,\ell)}
\in
C^{m-1}(\overline{\mathbb D})$.

\begin{theorem}[Higher-order boundary reconstruction]
\label{thm:higher-boundary-reconstruction}
For every $m\ge1$ and every $0\le q\le L-1$,
the following normalized radial boundary limit exists and satisfies
\begin{equation*}\label{eq:higher-boundary-reconstruction}
\begin{aligned}
\kappa_{m,q}
&:=
\lim_{r\to1^-}
\frac{
\mathcal R_m^{(m-1)}(r\omega^{-q})
}{
\log\frac{1}{1-r}
}
=
\omega^{q(m-1)}\widehat{V}_{m,q},
\\[5pt]
\widehat{V}_{m,q}
&=
\omega^{-q(m-1)}\kappa_{m,q},
\\[5pt]
V_m^{(k,\ell)}(j)
&=
\sum_{q=0}^{L-1}
\kappa_{m,q}\,
\omega^{q(j-m+1)}.
\end{aligned}
\end{equation*}
\end{theorem}

\begin{proof}
Differentiate the representation of $\mathcal R_m$ exactly $m-1$
times.  Since
$\mathcal E_m^{(k,\ell)}
\in
C^{m-1}(\overline{\mathbb D})$,
its $(m-1)$-st derivative remains bounded on radial approach to the
unit circle.

Fix $q$.  Along
$t=r\omega^{-q}$,
the mode with Fourier index $p=q$ satisfies
\[
\begin{aligned}
&
\left.
\frac{d^{m-1}}{dt^{m-1}}
\operatorname{Li}_m(\omega^qt)
\right|_{t=r\omega^{-q}}=
\omega^{q(m-1)}
\log\frac1{1-r}
+
O(1)
\end{aligned}
\]
by Lemma~\ref{lem:polylog-boundary-reconstruction} and the chain rule.

For $p\neq q$, the argument tends to $\omega^{p-q}\neq1$,
so the corresponding derivative remains bounded.  Hence
\[
\mathcal R_m^{(m-1)}(r\omega^{-q})
=
\omega^{q(m-1)}
\widehat{V}_{m,q}
\log\frac1{1-r}
+
O(1).
\]
Dividing by
$\log\frac1{1-r}$
and letting \(r\to1^-\) gives
\[
\kappa_{m,q}
=
\omega^{q(m-1)}\widehat V_{m,q}.
\]
Hence
\[
\widehat V_{m,q}
=
\omega^{-q(m-1)}\kappa_{m,q},
\]
and the final reconstruction formula follows from the finite Fourier
inversion formula in \eqref{eq:higher-Fourier-expansion}.
\end{proof}

Together with Theorem~\ref{thm:polylogarithmic-hierarchy}, the results of this section establish a two-way correspondence between the periodic asymptotic hierarchy and the boundary data of Yang's generating function.

\medskip

\vspace{20pt}
\noindent\textbf{AI Disclosure.}
The research questions and the main mathematical ideas of this work originated with the authors. ChatGPT 5.6 Sol (OpenAI) was used as an auxiliary tool during the development of the manuscript to explore and check some intermediate calculations and arguments, and to assist with the organization, exposition, and language editing of the text. All AI-assisted mathematical content was independently reviewed and verified by the authors. The authors take full responsibility for the results, proofs, references, and final text of the article.

\vspace{20pt}
\noindent\textbf{Acknowledgment.}
Y. Liu is supported by the Natural Science Foundation Project in Henan Province (No. 262300421861), the funding program for young backbone teachers in higher education institutions in Henan Province (No. 2024GGJS106), the key research projects of higher education institutions in Henan Province (No. 25B110009) and the general project cultivation fund of Nanyang Normal University (No. 2025PY034). Y. Lu is supported by NSFC (Grant Nos. 12031002 and 12671149). Y. Yang  is supported by NSFC (Grant No.12471117).

\vspace{20pt}
%
\noindent\textbf{Conflict of interest}
The authors declare that they have no conflict
of interest. 

\vspace{20pt}

\noindent\textbf{Data availability statement}
No data, models, or code were generated or used for the research described in the article.


\begin{thebibliography}{99}

\bibitem{AzariLuYang2017}
F. Azari Key, Y. Lu and R. Yang,
On numerical invariants for homogeneous submodules in
$H^2(\mathbb D^2)$,
\emph{New York J. Math.} \textbf{23} (2017),
505--526.

\bibitem{ChenGuo2003}
X. Chen and K. Guo,
\emph{Analytic Hilbert Modules},
Chapman \& Hall/CRC Research Notes in Mathematics, vol.~433,
Chapman \& Hall/CRC, Boca Raton, FL, 2003.

\bibitem{DouglasPaulsen1989}
R. G. Douglas and V. I. Paulsen,
\emph{Hilbert Modules over Function Algebras},
Pitman Research Notes in Mathematics Series, vol.~217,
Longman Scientific \& Technical, Harlow, 1989.

\bibitem{GuoWang2012}
K. Guo and P. Wang,
Essentially normal Hilbert modules and $K$-homology IV:
quasi-homogeneous quotient modules of Hardy module on the polydisks,
Sci. China Math. \textbf{55} (2012), no.~8, 1613--1626.


\bibitem{GuoWangZhang2012}
K. Guo, K. Wang and G. Zhang,
Trace formulas and $p$-essentially normal properties of quotient modules
on the bidisk,
J. Operator Theory \textbf{67} (2012), no.~2, 511--535.


\bibitem{GuoYang2004}
K. Guo and R. Yang,
The core function of submodules over the bidisk,
\emph{Indiana Univ. Math. J.} \textbf{53} (2004), no.~1,
205--222.


\bibitem{LiuLuZu2026Quadratic}
Y. Liu, Y. Lu and C. Zu,
On numerical invariants for submodules $[(z-w)^2]$ in
$H^2(\mathbb D^2)$,
\emph{Bull. Malays. Math. Sci. Soc.} \textbf{49} (2026),
Article~134.


\bibitem{LiuLuZu2026Block}
Y. Liu, Y. Lu and C. Zu,
Block repetition of numerical invariants for the submodules
$[z^k-w^k]$ in $H^2(\mathbb D^2)$,
arXiv:2608.13275, 2026.

\bibitem{LiuLuZu2026Strict}
Y. Liu, Y. Lu and C. Zu,
Strict monotonicity of numerical invariants for the submodules
$[(z-w)^k]$ in $H^2(\mathbb D^2)$,
arXiv:2608.17780, 2026.

\bibitem{LiuLuZu2026Quasi}
Y. Liu, Y. Lu and Y. Yang,
Numerical invariants and parameter recovery for
quasi-homogeneous submodules $[z^k-w^\ell]$, arXiv:2609.25640, 2026.

\bibitem{LuZu2026CMV}
Y. Lu and C. Zu,
Verblunsky coefficients, CMV matrices and numerical invariants
of homogeneous bidisc submodules,
arXiv:2608.18456, 2026.

\bibitem{NIST2010}
F. W. J. Olver, D. W. Lozier, R. F. Boisvert and C. W. Clark, eds.,
\emph{NIST Handbook of Mathematical Functions},
Cambridge University Press, Cambridge, 2010.

\bibitem{Yang2001}
R. Yang,
Operator theory in the Hardy space over the bidisk. III,
\emph{J. Funct. Anal.} \textbf{186} (2001), no.~2,
521--545.


\bibitem{Yang2005HS}
R. Yang,
Hilbert-Schmidt submodules and issues of unitary equivalence,
\emph{J. Operator Theory} \textbf{53} (2005), no.~1,
169--184.

\bibitem{Yang2005Core}
R. Yang,
The core operator and congruent submodules,
\emph{J. Funct. Anal.} \textbf{228} (2005), no.~2,
469--489.


\end{thebibliography}
\end{document}